\documentclass[12pt,a4paper]{amsart}

\usepackage{amssymb}
\usepackage{multirow}
\usepackage{graphicx}
\usepackage{amsmath,amsbsy,amsfonts,amsthm,latexsym,amsopn,amstext,amsxtra,euscript,amscd,mathrsfs}
\usepackage{color}
\usepackage{multicol}

\theoremstyle{plain}
\newtheorem{theorem}{Theorem}[section]
\newtheorem{corollary}[theorem]{Corollary}
\newtheorem{lemma}[theorem]{Lemma}

\theoremstyle{definition}

\theoremstyle{remark}
\newtheorem{remark}[theorem]{Remark}

\numberwithin{equation}{section}

\numberwithin{table}{section}

\numberwithin{figure}{section} 
\newcommand{\diag}{\operatorname{diag}}
\usepackage{mathdots}

\begin{document}
\title[Unipotent Schwarz matrices]{
$a$-potent Schwarz matrices and Bessel-like Jacobi polynomials}

\author{Alexander Dyachenko}
\address{Keldysh Institute of Applied Mathematics, Russian Academy of Sciences, 125047, Moscow,  Russia}
\email{diachenko@sfedu.ru}

\author{Carlos M. da Fonseca}
\address{Kuwait College of Science and Technology, Doha District, Safat
	13133, Kuwait}
\address{Chair of Computational Mathematics, University of Deusto, 48007 Bilbao, Spain}
\email{c.dafonseca@kcst.edu.kw}

\author{Mikhail Tyaglov}
\address{Department of Mathematics and Computer Sciences, St.Petersburg State University,
St.Petersburg, 199178, Russia}
\email{tyaglov@mail.ru}

\date{\today}

\begin{abstract}
We consider the problem of the reconstruction of a Schwarz matrix
from exactly one given eigenvalue. This inverse eigenvalue problem leads to the Jacobi orthogonal polynomials~$\{P_k^{(-n,n)}\}_{k=0}^{n-1}$ that can be treated as a discrete finite analogue of Bessel polynomials.
\end{abstract}

\maketitle

\unitlength=1mm

\section{Introduction}

The matrix inverse eigenvalue problem consists in reconstructing a specific structured matrix from a prescribed spectral data. This topic is of great interest in different
branches of mathematics, continually providing new and surprising results (see for
example~\cite{C1998,JSW2009,SJA2021,S2019} and the references there).



From the celebrated work~\cite{W1945} by Hubert S. Wall, it follows that there exists a \emph{unique} matrix
\begin{equation}\label{main.matrix}
J_{n}=
\begin{bmatrix}
    -b_0 & 1 & & & \\
    -b_1 & 0 & 1 & & \\
     &-b_2 & \ddots & \ddots & \\
     & & \ddots & \ddots & 1\\
     & & & -b_{n-1} & 0\\
\end{bmatrix}
\end{equation}
with $b_k>0$, for $k=0,1,\ldots,n-1$, whose characteristic polynomial is~$(x+1)^{n}$ (see also~\cite{S1956, T2012, W1948} for more details). Wall's argumentation is based on the stability of the characteristic polynomial, i.e.\ on that all its zeros have negative real parts --- which is of course true for~$(x+1)^{n}$.

By analogy with nilpotent matrices, one introduces \textit{unipotent} matrices: an $n\times n$ matrix is called unipotent if its characteristic polynomial is $(x-1)^n$, or, equivalently, if all its eigenvalues are equal to $1$. Various disciplines deal with unipotent matrices, cf.\ e.g.~\cite{CHJ2013,HXHY2019,M1998,Z1969}. We consider an extension of this definition, namely  matrices with only a single eigenvalue labelled by~$a$. In brief, we are looking for an ``$a$-potent'' matrix of the form~\eqref{main.matrix} where $a\in\mathbb{C}\setminus\{0\}$.

Matrices of the form \eqref{main.matrix} are called Schwarz matrices.
Such matrices are very important in automatic control theory and signal processing (see, for example,~\cite{CH2004}), and several general inverse problems have been previously considered in~\cite{Behn_et_al,CdFP2024,DJOvdD2000,EH2003,Garnett_Shader}. So, here we consider the inverse eigenvalue problem for a
Schwarz matrix whose characteristic polynomial is
\begin{equation}\label{main.poly}
p(z)=(z-a)^{n}=\sum_{i=0}^{n}(-1)^{n-i}\binom{n}{i}a^{n-i}z^i,
\end{equation}
where $a$ is a given nonzero complex number. That is, we want to construct a Schwarz matrix~\eqref{main.matrix} with only one eigenvalue $a$.

Since $J_k$ is the leading principal $k\times k$ submatrix of
$J_{n}$, with $k\in\{1,\ldots,n\}$, the polynomials
\begin{equation*}
P_k(z)=\det(zI_k-J_k)\, ,\quad \mbox{for}\; k=1,\ldots,n\, ,
\end{equation*}
where $I_k$ is the identity matrix of order $k$ and
$P_{n}(z)=p(z)$, satisfy the following three-term recurrence
relation:
\begin{eqnarray}
P_0(z) & = & 1\, , \nonumber\\
P_1(z) & = & z+b_0\, , \label{three-term.rec.rel.general}\\
P_{k+1}(z) & = & zP_k(z)+b_kP_{k-1}(z)\, ,\quad \mbox{for}\; k=1,\ldots,n-1\, .
\nonumber
\end{eqnarray}

By the Favard theorem, see%
\footnote{In~\cite{MA-N2001}, this extended version is attributed to Chihara, although in essence
it already existed in~\cite{Geronimus1940}.}
~\cite[pp.\ 21--22]{C1978}, the
sequence $\{ P_k(z)\}_{k=0}^{n-1}$ is orthogonal w.r.t. a linear
functional. Observe that, given arbitrary~$n>1$ and~$a\in\mathbb{C}\setminus\{0\}$, this functional is not positive and cannot be written as an integral w.r.t.\ a measure on the real line (otherwise~$J_n$ would have $n$ distinct real eigenvalues). Accordingly, for~$n>1$ the numbers~$b_1,\dots,b_{n-1}$ in~\eqref{three-term.rec.rel.general} cannot be chosen negative while~$b_0$ remains real.
As is seen from Section \ref{subsection:orthogonality}, the support of the orthogonality functional is natural to take consisting of a unique point $a$.

We find the aforementioned matrix~$J_n$, see the values of~$b_0$ and~$b_1,\dots,b_{n-1}$ in, resp.,~\eqref{coeff.0} and~\eqref{coeff.k.2}. The sequence $\{ P_k(z)\}_{k=0}^{n-1}$ satisfying~\eqref{three-term.rec.rel.general} reduces to a finite chunk of the Jacobi polynomials with non-standard parameters.
The explicit expression of the polynomial~$P_k(z)$ for~$a=1$ is given in~\eqref{eq:P.explicit.2F1}, while the relation of the polynomials with distinct~$a$ is in~\eqref{P.with.different.a}.
%

It turns out that three-term recurrence relations for the Bessel
polynomials~\cite{KrallFrink.1948} have a form similar to~\eqref{three-term.rec.rel.general},
so Bessel polynomials are also characteristic polynomials
of Schwarz matrices, and for $a=-\dfrac{1}{n}$ the polynomials~$P_{k}(z)$
approximate Bessel polynomials uniformly
on the compact sets of the complex plane as $n\to+\infty$. Moreover,
the support of the functional of orthogonality for Bessel
polynomials on the real line may also be treated as consisting of one point of essential singularity.
Thus, we can treat polynomials~$\{P_k(z)\}_{k=0}^{n-1}$ as
a discrete finite analogue of Bessel polynomials.

The paper is organised as follows. Section~\ref{Section:matrix} is devoted to finding the entries of the
matrix~$J_{n}$. In Sections~\ref{sect:P_k.basic.properties} and~\ref{subsection:orthogonality}, we
investigate properties of the orthogonal polynomials~$P_k(z)$
and find their functional of orthogonality on the real line.
In Section~\ref{section:Bessel}, we compare our polynomials with
Bessel polynomials and provide some calculations showing that
the zeroes of $P_{k}(z)$ approach the zeroes of Bessel polynomials
for $a=-\dfrac1n$ as $n\to+\infty$.
In Section~\ref{section:Zeroes}, we calculate zeroes of $P_{k}(z)$ for varying~$a$ and~$k$ to
estimate numerically their asymptotic behaviour as~$n$ is large.

\section{Explicit formula for Schwarz matrix}\label{Section:matrix}

To find the entries of the matrix $J_{n}$ or, similarly, the
coefficients of the three-term recurrence
relations~\eqref{three-term.rec.rel.general}, we consider the
rational function
\begin{equation}\label{rat.func.general}
R(z)=b_0((zI_{n}-J_{n})^{-1}e_1,e_1)=\dfrac{q(z)}{p(z)},
\end{equation}
where $e_1$ is the first coordinate vector. Here the polynomial
$q(z)$, multiplied by $b_0$, is the characteristic polynomial of the
principal submatrix of $J_{n}$~\eqref{main.matrix} constructed
from $J_{n}$ by deleting its first row and the first column:
\begin{equation*}
q(z)=b_0\det
\begin{bmatrix}
     z & -1 &  0 & & \\
    b_2 & z & -1 & &  \\
      &b_3 & \ddots &\ddots& \\
      & &\ddots& \ddots & -1\\
     & & & b_{n-1}&z\\
\end{bmatrix}.
\end{equation*}

Observe that the polynomial $p(z)$, the characteristic polynomial
of the matrix $J_{n}$, can be represented as follows
\begin{equation}\label{even.odd.parts}
p(z)=Q(z)+q(z)\, ,
\end{equation}
where
\begin{equation*}
Q(z)=\det
\begin{bmatrix}
     z & -1 &  0 & & \\
    b_1 & z & -1 & &  \\
      &b_2 & \ddots &\ddots& \\
      & &\ddots& \ddots & -1\\
     & & & b_{n-1} & z\\
\end{bmatrix}\, .
\end{equation*}
It is easy to see that $q(z)$ is an even polynomial whenever $n$ is
odd, and is an odd polynomial otherwise, while the
polynomial $Q(z)$ is odd if $n$ is odd, and is even if $n$ is even.
Therefore, by~\eqref{even.odd.parts}, the polynomial $q(z)$ is the
even (odd) part of the polynomial~$p(z)$ whenever $n$ is odd (even). This
means that $q(z)$ can be obtained from $p(z)$ in the following way
\begin{equation*}\label{poly.q}
q(z)=\dfrac{p(z)-(-1)^{n}p(-z)}{2}\,,
\end{equation*}
so that
\begin{eqnarray}
q(z) & = &
-\binom{n}{1}az^{n-1}-\binom{n}{3}a^3z^{n-3}-\binom{n}{5}a^5z^{n-5}-\cdots
\nonumber
\\ [3pt]
& = &
-\sum_{j=0}^{\left\lfloor\tfrac{n-1}2\right\rfloor}\binom{n}{2j+1}a^{2j+1}z^{n-2j-1}.
\label{second.poly}
\end{eqnarray}
For the calculation of the entries of the matrix $J_{n}$ we first study the rational
function~$R(z)$ defined in~\eqref{rat.func.general}.

\begin{theorem}\label{Theorem.rat.func}
The rational function $R(z)=\dfrac{q(z)}{p(z)}$ has the form
\begin{equation*}\label{rat.func}
R(z)=\dfrac{\alpha_1}{z-a}+\dfrac{\alpha_1}{(z-a)^2}+\cdots+\dfrac{\alpha_n}{(z-a)^{n}}=
\dfrac{s_0}{z}+\dfrac{s_1}{z^2}+\dfrac{s_2}{z^3}+\cdots,
\end{equation*}
where
\begin{equation}\label{alpha_k}
\alpha_k=-2^{k-1}a^{k}\binom{n}{k}\, , \quad \mbox{for}\; k=1,\ldots,n\, ,
\end{equation}
and
\begin{equation}\label{s_k}
s_m=\sum\limits_{k=1}^{n}a^{m-k+1}\binom{m}{k-1}\alpha_{k}=-a^{m+1}\sum\limits_{k=1}^{n}2^{k-1}\binom{m}{k-1}\binom{n}{k}.
\end{equation}
\end{theorem}
\begin{proof}
Expanding the polynomial $q(z)$ into the Taylor series at the point
$a$, we obtain
\begin{equation*}\label{Theorem.rat.func.proof.1}
R(z)=\dfrac{q(a)}{(z-a)^{n}}+\dfrac{q'(a)}{1!(z-a)^{n-1}}+\dfrac{q''(a)}{2!(z-a)^{n-2}}+\cdots+\dfrac{q^{(n-1)}(a)}{(n-1)!(z-a)}\,,
\end{equation*}
so
\begin{equation*}\label{alpha_k.2}
\alpha_k=\dfrac{q^{(n-k)}(a)}{(n-k)!}\, \, , \quad \mbox{for}\; k=1,\ldots,n\, .
\end{equation*}
From~\eqref{second.poly}, we have
\begin{equation*}\label{Theorem.rat.func.proof.2}
\dfrac{q^{(n-k)}(z)}{(n-k)!}=-\sum_{j=1}^{\left\lfloor\tfrac{k-1}2\right\rfloor}
\binom{n}{2j+1}\binom{n-2j-1}{n-k}a^{2j+1}z^{k-2j-1}\, , \quad \mbox{for}\;
k=1,\ldots,n,
\end{equation*}
which implies
\begin{equation}\label{Theorem.rat.func.proof.3}
\alpha_{k}=\dfrac{q^{(n-k)}(a)}{(n-k)!}=-a^{k}\binom{n}{k}\sum_{j=0}^{\left\lfloor\tfrac{k-1}2\right\rfloor}\binom{k}{2j+1}\, , \quad \mbox{for}\;
k=1,\ldots,n.
\end{equation}

Now substituting
\begin{equation*}\label{Theorem.rat.func.proof.8}
2^{k-1}=\sum_{j=0}^{\left\lfloor\tfrac{k-1}2\right\rfloor}\binom{k}{2j+1}
\end{equation*}
into~\eqref{Theorem.rat.func.proof.3}, we get~\eqref{alpha_k}.

To establish~\eqref{s_k}, we recall that, for any nonnegative integer
number $\ell$,
\begin{equation*}\label{Theorem.rat.func.proof.9}
\displaystyle\dfrac1{(z-a)^{\ell}}=\sum_{m=0}^{\infty}\binom{m+\ell-1}{\ell-1}\dfrac{a^m}{z^{m+\ell}}\, ,
\end{equation*}
so
\begin{eqnarray*}
R(z) & = & \sum_{k=1}^n\dfrac{\alpha_k}{(z-a)^{k}}
= 
\sum_{m=0}^{\infty}\sum_{k=1}^n\dfrac{\binom{m+k-1}{k-1}a^m\alpha_k}{z^{m+k}}
\\ & = &
\sum_{m=0}^{\infty} \dfrac1{z^{m+1}}\sum_{k=1}^{n}a^{m-k+1}\binom{m}{k-1}\alpha_{k}\, .
\end{eqnarray*}
This formula together with~\eqref{alpha_k} gives us~\eqref{s_k}.
\end{proof}

\begin{remark}\label{Remark:moments}
Note that the numbers $s_m$ satisfy the following recurrence
relation
$$
(1+m)a^2s_m+2nas_{m+1}-(3+m)s_{m+2}=0\, .
$$
A careful consideration shows that the moments $s_m$ are polynomials in $n$, and can be
represented as
$$
s_m(n)=-\dfrac{a^{m+1}n}{(m+1)!}\,p_m(n)\, , \quad \mbox{for}\; m=0,1,2,\ldots,
$$
where $p_m(x)$ are the polynomials satisfying the following three-term
recurrence relations:
$$
p_{m+1}(x)=xp_m(x)+m(m+1)p_{m-1}(x)\, , \quad \mbox{for}\; m=1,2,\ldots
$$
with the initial conditions $p_0(x)=1$, $p_1(x)=x$. These recurrence relations generate the following
continued fraction (see, e.g.,~\cite[p.231]{Lorentzen}):
$$
\Phi(x)=x\int\limits_{0}^{\infty}e^{-xt}\tanh(t)\, dt=\cfrac{1}{x+\cfrac{1\cdot2}{x+\cfrac{2\cdot3}{x+\cdots}}}
$$
convergent for $x>0$. In fact, the integral itself converges in the right half-plane, so it generates a measure
of orthogonality for the polynomials~$p_m(x)$ with the support on the imaginary axis.

It is not hard to identify this system of orthogonal polynomials: on comparing the
recurrence coefficients one observes that they are the Meixner--Pollaczek
polynomials~$P_m^{(1)}\big(\:\cdot\:;\frac\pi2\big)$ up to rotation and renormalisation.
Namely,
\[
    p_m(n)
    =
    -\frac {(m+1)!}{a^{m+1} n} s_m(n)
    =
    (m+1)!\:
    {}_2F_1\left(\begin{matrix}-m,1-n\\2\end{matrix}\,\Big|\,2\right)
    = \frac{m!}{i^{m}} P_m^{(1)}\left(in;\frac{\pi}2\right),
\]
so the corresponding orthogonality measure is~$\left|\Gamma(1-x)\right|^2$
with~$x\in i\mathbb{R}$,
see\footnote{Note a typo in formula~(3.22) on that page: the third parameter of~${}_2F_1$ is
    actually~$2\lambda$, not~$2$.}~\cite[p.~180]{C1978}.

\end{remark}

\vspace{2mm}

The entries $b_m$ of the matrix $J_{n}$ defined in~\eqref{main.matrix} can be found by the formula (see, e.g.,~\cite{HT2012, T2012} and references therein)
\begin{equation}\label{coeff.b_k}
b_m=-\dfrac{D_{m-1}(R)D_{m+1}(R)}{D_m^2(R)}\, , \quad \mbox{for}\; m=1,2,\ldots,n-1\, ,
\end{equation}
where $D_0(R)=1$, and $D_{m}(R)$, for $m=1,\ldots,n$, are the Hankel determinants defined as follows
\begin{equation}\label{Hankel.determinants.1}
D_{m}(R)=
\begin{vmatrix}
    s_0 &s_1 &s_2 &\cdots &s_{m-1}\\
    s_1 &s_2 &s_3 &\cdots &s_{m}\\
    \vdots&\vdots&\vdots&\cdots&\vdots\\
    s_{m-1} &s_{m} &s_{m+1} &\cdots &s_{2m-2}
\end{vmatrix}\, , \quad \mbox{for}\; m=1,2,\dots,n.
\end{equation}
Recall that
\begin{equation*}
D_{m}(R)=0\, ,\qquad\mbox {for $m=n+1,n+2,\ldots,$}
\end{equation*}
since the rational function $R(z)$ has exactly $n$ poles, counting
multiplicities (see, e.g.,~\cite[Ch. XV, \S10, Theorem 8]{Gantmakher.1} or, alternatively,~\cite[Theorem 1.2]{HT2012}).

Moreover,
\begin{equation}\label{coeff.0}
b_0=-an,
\end{equation}
since $-b_0$ equals the trace of $J_{n}$.

To find the exact formula for the minors $D_{m}(R)$ we need the
following general result that can be obtained using elementary
determinant properties.

\begin{lemma}\label{Lemma.Hankel.minors.equivalency}
The minors $D_{m}(R)$ can be represented in the following form:
\begin{equation}\label{Hankel.determinants.2}
D_{m}(R)=
\begin{vmatrix}
    \alpha_1 &\alpha_2 & \dots &\alpha_{m}\\
    \alpha_2 &\alpha_3 &\dots &\alpha_{m+1}\\
    \vdots&\vdots&\ddots&\vdots\\
    \alpha_{m} &\alpha_{m+1} &\dots &\alpha_{2m-1}
\end{vmatrix}\, ,
\end{equation}
for $m=1,2,\dots,n$, where we set $\alpha_j=0$, for $j>n$.
\end{lemma}
\begin{proof}
Indeed, the $k$th column $\mathbf{C}_k$ of the minor $D_{m}(R)$ has
the form
\begin{equation}\label{Lemma.Hankel.minors.proof.1}
\mathbf{C}_k=
\begin{bmatrix}
    s_{k-1}\\
    \vdots\\
    s_{k+m-2}
\end{bmatrix}\, , \quad \mbox{for}\; k=1,\dots,m\, .
\end{equation}

If we change columns as follows
\begin{equation}\label{Lemma.Hankel.minors.proof.2}
\widetilde{\mathbf{C}}_k=\sum_{j=1}^{k}(-a)^{k-j}\binom{k-1}{j-1}\mathbf{C}_{j}\, ,
\end{equation}
for $k=1,\dots,m$, where $\widetilde{\mathbf{C}}_k$ is the $k$th new column, then the minor $D_{m}(R)$ does not change its value.
From~\eqref{Lemma.Hankel.minors.proof.1}--\eqref{Lemma.Hankel.minors.proof.2},
we obtain that the $\ell$-entry $(\widetilde{\mathbf{C}}_k)_\ell$ of
the new $k$th column has the form
%
%
\begin{eqnarray*}
(\widetilde{\mathbf{C}}_k)_\ell & = &
\sum_{j=1}^{k}(-a)^{k-j}\binom{k-1}{j-1}s_{j+\ell-2} \\ & = & \sum_{j=1}^{k}(-a)^{k-j}\binom{k-1}{j-1}\sum\limits_{i=1}^{n}a^{j+\ell-1-i}\binom{j+\ell-2}{i-1}\alpha_{i}
\\ [3pt]
& = &
\sum\limits_{i=1}^{n}a^{k+\ell-1-i}\alpha_{i}\sum_{j=0}^{k-1}(-1)^{k-1-j}\binom{k-1}{j}\binom{j+\ell-1}{i-1}
\\ [3pt]
& = &
\sum_{i=1}^{n}a^{k+\ell-1-i}\binom{\ell-1}{i-k}\alpha_{i} 
=\sum_{i=1}^{\ell}a^{\ell-i}\binom{\ell-1}{i-1}\alpha_{k+i-1}\, ,
\end{eqnarray*}
where $k,\ell=1,\dots,m$ and $\alpha_j=0$, for $j>n$. Here we used the identity
\begin{equation*}\label{identity.1}
\sum_{k=0}^m(-1)^{m-k}\binom{m}{k}\binom{k+\ell}{i}=\binom{\ell}{i-m}
\end{equation*}
which is true for any nonnegative integer numbers $m,\ell,i$ (cf., e.g.,~\cite[Section 4.2.5, no. 55]{PBM2002} or~\cite[Chapter 5, Table 169]{GKP1998}).

Thus, we get the following formula for the minors $D_{m}(R)$:
\begin{equation}\label{Lemma.Hankel.minors.proof.4}
D_{m}(R)=
\begin{vmatrix}
    \alpha_1 &\alpha_2  &\cdots &\alpha_m\\
    a\alpha_1+\alpha_2 & a\alpha_2+\alpha_3 &\cdots &a\alpha_m+\alpha_{m+1}\\
    \vdots&\vdots&\ddots&\vdots\\
    \displaystyle\sum_{r=1}^{m}a^{m-r}\binom{m}{r}\alpha_{r} &\displaystyle\sum_{r=1}^{m}a^{m-r}\binom{m}{r}\alpha_{r+1}
    &\cdots &\displaystyle\sum_{r=1}^{m}a^{m-r}\binom{m}{r}\alpha_{r+m-1}
\end{vmatrix} \, ,
\end{equation}
for $m=1,2,\dots,n$.

Let us consider now the $\ell$th row of the matrix~\eqref{Lemma.Hankel.minors.proof.4}. It has the form
\begin{equation}\label{Lemma.Hankel.minors.proof.5}
\mathbf{R}_\ell=
\begin{bmatrix}
    v_{\ell,1}, v_{\ell,2},\ldots, v_{\ell,m}
\end{bmatrix}\, \, , \quad \mbox{for}\; \ell=1,\dots,m\, ,
\end{equation}
where
\begin{equation}\label{Lemma.Hankel.minors.proof.6}
v_{\ell k}=\sum_{i=1}^{l}a^{\ell -i}\binom{\ell-1}{i-1}\alpha_{i+k-1}\, , \quad \mbox{for}\;
\ell,k=1,2,\ldots,m\, .
\end{equation}

Now we change the rows as follows
\begin{equation}\label{Lemma.Hankel.minors.proof.7}
\widetilde{\mathbf{R}}_\ell=\sum_{j=1}^{\ell}(-a)^{\ell-j}\binom{\ell-1}{j-1}\mathbf{R}_{j}\, , \quad \mbox{for}\;
\ell=1,\dots,m\, ,
\end{equation}
where $\widetilde{\mathbf{R}}_\ell$ is the $\ell$th new row. This
operation does not change the value of the minor $D_{m}(R)$.
From~\eqref{Lemma.Hankel.minors.proof.5}--\eqref{Lemma.Hankel.minors.proof.7},
we conclude that the $k$th entry $(\widetilde{\mathbf{R}}_\ell)_k$
of the new $\ell$th row has the form
\begin{eqnarray}\label{Lemma.Hankel.minors.proof.8}
(\widetilde{\mathbf{R}}_\ell)_k & = & \sum_{j=1}^{\ell}(-a)^{\ell-j}\binom{\ell-1}{j-1}v_{j,k}
\nonumber
\\ [3pt]
& = &
\sum\limits_{j=1}^{\ell}\sum\limits_{i=1}^{j}a^{\ell-i}(-1)^{\ell-j}
\binom{\ell-1}{j-1}\binom{j-1}{i-1}\alpha_{i+k-1}
\nonumber
\\ [2pt]
& = &
\sum\limits_{i=1}^{\ell}a^{\ell-i}\alpha_{i+k-1} \sum\limits_{j=i}^{\ell}(-1)^{\ell-j}\binom{\ell-1}{j-1}\binom{j-1}{i-1}
\nonumber
\\ [2pt]
& = &
\sum\limits_{i=1}^{\ell}a^{\ell-i}\alpha_{i+k-1}\delta_{\ell i}=\alpha_{k+\ell-1},
\end{eqnarray}
for $k,\ell=1,\dots,m$.
Now~\eqref{Lemma.Hankel.minors.proof.8} implies~\eqref{Hankel.determinants.2}.
\end{proof}

Now we are in a position to calculate the values of the Hankel minors $D_m(R)$.

\begin{theorem}\label{Theorem.Hankel.minors.exact}
The minors $D_m(R)$ have the form
\begin{equation}\label{Hankel.determinants.3}
D_{m}(R)=(-1)^{\tfrac{m(m+1)}2}a^{m^2}2^{m(m-1)}\binom{n}{m}\prod_{k=0}^{m-1}
\dfrac{\binom{n+k}{n-k}}{\binom{m+k}{m-k}}\, ,
\end{equation}
for $m=1,2,\dots,n$.
\end{theorem}

\begin{proof}
From~\eqref{alpha_k} and~\eqref{Hankel.determinants.2} we have
\begin{eqnarray}
D_{m}(R) & = & \begin{vmatrix}
    -a\binom{n}{1} &-2a^2\binom{n}{2} &\dots &-2^{m-1}a^{m}\binom{n}{m}  \\
    -2a^2\binom{n}{2} &-2^2a^3\binom{n}{3} &\dots &-2^{m}a^{m+1}\binom{n}{m+1}\\
    \vdots&\vdots & \ddots&\vdots\\
    -2^{m-1}a^{m}\binom{n}{m} &-2^{m}a^{m+1}\binom{n}{m+1} &\dots &-2^{2m-2}a^{2m-1}\binom{n}{2m-1}
\end{vmatrix} \nonumber \\ [0.2cm]
& = &  (-1)^{m}2^{\tfrac{m(m-1)}2}a^{\tfrac{m(m+1)}2}
\begin{vmatrix}
    \binom{n}{1} &2a\binom{n}{2} &\dots &2^{m-1}a^{m-1}\binom{n}{m}\\
    \binom{n}{2} &2a\binom{n}{3} &\dots &2^{m-1}a^{m-1}\binom{n}{m+1}\\
    \vdots&\vdots&\ddots&\vdots\\
    \binom{n}{m} &2a\binom{n}{m+1} &\dots &2^{m-1}a^{m-1}\binom{n}{2m-1}
\end{vmatrix} \nonumber \\ [0.2cm]
& = &
(-1)^{m}2^{\tfrac{m(m-1)}2}a^{\tfrac{m(m+1)}2}(2a)^{\tfrac{m(m-1)}2}
\begin{vmatrix}
    \binom{n}{1} &\binom{n}{2}  &\dots &\binom{n}{m}\\
    \binom{n}{2} &\binom{n}{3} &\dots &\binom{n}{m+1}\\
    \vdots&\vdots&\ddots&\vdots\\
    \binom{n}{m} &\binom{n}{m+1} &\dots &\binom{n}{2m-1}
\end{vmatrix} \nonumber
\\ [0.2cm]
& = & (-1)^{m}2^{m(m-1)}a^{m^2}T_{m}\, ,
\label{Theorem.Hankel.minors.exact.proof.1}
\end{eqnarray}
where
\begin{eqnarray}
T_{m} & = & \begin{vmatrix}
    \binom{n}{1} &\binom{n}{2} &\dots &\binom{n}{m}\\
    \binom{n}{2} &\binom{n}{3} &\dots &\binom{n}{m+1}\\
    \vdots&\vdots&\ddots&\vdots\\
    \binom{n}{m} &\binom{n}{m+1} &\dots &\binom{n}{2m-1}
\end{vmatrix} \nonumber
\\ [0.2cm] & = &
\begin{vmatrix}
    \tfrac{n}{1!} &\tfrac{n(n-1)}{2!} &\dots &\tfrac{n(n-1)\cdots(n-m+1)}{m!}\\
    \tfrac{n(n-1)}{2!} &\tfrac{n(n-1)(n-2)}{3!} &\dots &\tfrac{n(n-1)\cdots(n-m)}{(m+1)!}\\
    \vdots&\vdots&\ddots&\vdots\\
    \tfrac{n(n-1)\cdots(n-m+1)}{m!} &\tfrac{n(n-1)\cdots(n-m)}{(m+1)!} &\dots &\tfrac{n(n-1)\cdots(n-2m+2)}{(2m-1)!}
\end{vmatrix} \label{Theorem.Hankel.minors.exact.proof.2}
\end{eqnarray}
for $m=1,2,3,\dots,n$. Taking into account that the $j$th, $(j+1)$th,..., $m$th rows have the
common factor $n(n-1)\cdots(n-j+1)$, for $j=1,\ldots,m$, we get
\begin{eqnarray*}
T_{m} & = & n^{m}(n-1)^{m-1}\cdots(n-m)^2(n-m+1)
\begin{vmatrix}
    \tfrac{1}{1!} &\tfrac{n-1}{2!} &\dots &\tfrac{(n-1)\cdots(n-m+1)}{m!}\\
    \tfrac{1}{2!} &\tfrac{n-2}{3!} & \dots &\tfrac{(n-2)\cdots(n-m)}{(m+1)!}\\
    \vdots&\vdots&\ddots&\vdots\\
    \tfrac{1}{m!} &\tfrac{n-m}{(m+1)!} &\dots &\tfrac{(n-m)\cdots(n-2m+2)}{(2m-1)!}
\end{vmatrix} \nonumber \\[0.2cm]
& = & \prod_{k=1}^{m}k!\binom{n}{k}\,
\begin{vmatrix}
    \tfrac{1}{1!} &\tfrac{n-1}{2!} &\dots &\tfrac{(n-1)\cdots(n-m+1)}{m!}\\
    \tfrac{1}{2!} &\tfrac{n-2}{3!} & \dots &\tfrac{(n-2)\cdots(n-m)}{(m+1)!}\\
    \vdots&\vdots&\ddots&\vdots\\
    \tfrac{1}{m!} &\tfrac{n-m}{(m+1)!} &\dots &\tfrac{(n-m)\cdots(n-2m+2)}{(2m-1)!}
    \label{Theorem.Hankel.minors.exact.proof.3}
\end{vmatrix}
\end{eqnarray*}
Now we add the $(m-1)$th column to the $m$th column, then the
$(m-2)$th column to the $(m-1)$th column, etc. The identity
\begin{equation}\label{Theorem.Hankel.minors.exact.proof.4}
\begin{array}{l}
\dfrac{(n-i)(n-i-1)\cdots(n-i-j+3)}{(i+j-2)!}+\dfrac{(n-i)(n-i-1)\cdots(n-i-j+2)}{(i+j-1)!}=\\
\\
=(n+1)\dfrac{(n-i)(n-i-1)\cdots(n-i-j+3)}{(i+j-1)!}
\end{array}
\end{equation}
is used repeatedly for each element of the $j$th column, for $j=m,m-1,\ldots,2$. Thus, we conclude
that the determinant does not change its value, while its $2$nd, $3$rd, ..., $m$th columns have the
common factor $n+1$. Therefore, we have
\begin{equation*}\label{Theorem.Hankel.minors.exact.proof.5}
T_{m}=(n+1)^{m-1}\prod_{k=1}^{m}k!\binom{n}{k}\,
\begin{vmatrix}
\tfrac{1}{1!}      &\tfrac{1}{2!}    &\dots &\tfrac{(n-1)\cdots(n-m+2)}{m!}\\
\tfrac{1}{2!} &\tfrac{1}{3!} &\dots &\tfrac{(n-2)\cdots(n-m+1)}{(m+1)!}\\
\vdots&\vdots&\ddots&\vdots\\
\tfrac{1}{m!} &\tfrac{1}{(m+1)!} &\dots &\tfrac{(n-m)\cdots(n-2m+3)}{(2m-1)!}
\end{vmatrix}
\end{equation*}

Now we use the similar procedure for the $3$rd, $4$th,$\ldots$,
$(m+1)$th columns. Namely, we add the $(m-1)$th column to the $m$th,
then the $(m-2)$th column to the $(m-1)$th, etc., and use the identity
\begin{equation*}\label{Theorem.Hankel.minors.exact.proof.6}
\begin{array}{l}
\dfrac{(n-i)(n-i-1)\cdots(n-i-j+4)}{(i+j-2)!}+\dfrac{(n-i)(n-i-1)\cdots(n-i-j+3)}{(i+j-1)!}=\\
\\
=(n+2)\dfrac{(n-i)(n-i-1)\cdots(n-i-j+4)}{(i+j-1)!}
\end{array}
\end{equation*}
repeatedly for each element of the $j$th column, for $j=3,\ldots,m$.
Again we conclude that the determinant does not change its value,
while its $3$rd, $4$th,$\ldots$, $m$th columns have the common
factor $n+2$. So we have
\begin{equation*}\label{Theorem.Hankel.minors.exact.proof.7}
T_{m}\displaystyle=(n+1)^{m-1}(n+2)^{m-2}\prod_{k=1}^{m}k!\binom{n}{k}\,
\begin{vmatrix}
    \tfrac{1}{1!}      &\tfrac{1}{2!}  &\cdots &\tfrac{(n-1)\cdots(n-m+3)}{m!}\\
    \tfrac{1}{2!} &\tfrac{1}{3!} &\cdots &\tfrac{(n-2)\cdots(n-m+2)}{(m+1)!}\\
    \vdots&\vdots &\ddots&\vdots\\
    \tfrac{1}{m!} &\tfrac{1}{(m+1)!} &\dots &\tfrac{(n-m)\cdots(n-2m+4)}{(2m-1)!}
\end{vmatrix}
\end{equation*}

Consequently, from a straightforward inductive argument, we get
\begin{eqnarray*}
T_{m} & = &\prod_{\ell
=1}^{m-1}(n+\ell)^{m-\ell}\prod_{k=1}^{m}k!\binom{n}{k}\,
\begin{vmatrix}
    \tfrac{1}{1!}  &\tfrac{1}{2!} &\cdots &\tfrac{1}{m!}\\
    \tfrac{1}{2!} &\tfrac{1}{3!}  &\cdots &\tfrac{1}{(m+1)!}\\
    \vdots&\vdots&\ddots&\vdots\\
    \tfrac{1}{m!} &\tfrac{1}{(m+1)!} &\cdots &\tfrac{1}{(2m-1)!}
\end{vmatrix} \nonumber
\\ [3pt]
&
=&\prod_{k=1}^{m}k!\binom{n}{k}\prod_{j=1}^{m-1}j!\binom{n+j}{j}
\, \begin{vmatrix}
    \tfrac{1}{1!}  &\tfrac{1}{2!} &\cdots &\tfrac{1}{m!}\\
    \tfrac{1}{2!} &\tfrac{1}{3!}  &\cdots &\tfrac{1}{(m+1)!}\\
    \vdots&\vdots&\ddots&\vdots\\
    \tfrac{1}{m!} &\tfrac{1}{(m+1)!} &\cdots &\tfrac{1}{(2m-1)!}
\end{vmatrix} \nonumber
\\  [3pt]
& = & m!\binom{n}{m}\prod_{k=1}^{m-1}(2k)!\binom{n+k}{n-k}
\begin{vmatrix}
    \tfrac{1}{1!}  &\tfrac{1}{2!} &\cdots &\tfrac{1}{m!}\\
    \tfrac{1}{2!} &\tfrac{1}{3!}  &\cdots &\tfrac{1}{(m+1)!}\\
    \vdots&\vdots&\ddots&\vdots\\
    \tfrac{1}{m!} &\tfrac{1}{(m+1)!} &\cdots &\tfrac{1}{(2m-1)!}
\end{vmatrix}\, . \label{Theorem.Hankel.minors.exact.proof.8}
\end{eqnarray*}

Note that the determinant $T_{m}$ equals $(-1)^{\tfrac{m(m-1)}2}$,
if $n=m$, as it follows from~\eqref{Theorem.Hankel.minors.exact.proof.2}. Therefore,
\begin{equation}\label{Theorem.Hankel.minors.exact.proof.9}
\begin{vmatrix}
    \tfrac{1}{1!}  &\tfrac{1}{2!} &\cdots &\tfrac{1}{m!}\\
    \tfrac{1}{2!} &\tfrac{1}{3!}  &\cdots &\tfrac{1}{(m+1)!}\\
    \vdots&\vdots&\ddots&\vdots\\
    \tfrac{1}{m!} &\tfrac{1}{(m+1)!} &\cdots &\tfrac{1}{(2m-1)!}
\end{vmatrix}=\displaystyle\dfrac{(-1)^{\tfrac{m(m-1)}2}}{m!\displaystyle\prod_{k=1}^{m-1}(2k)!\binom{m+k}{m-k}},
\end{equation}
and, consequently,
\begin{eqnarray*}
\displaystyle T_{m} & = &
(-1)^{\tfrac{m(m-1)}2}\dfrac{m!\displaystyle\binom{n}{m}\prod_{k=1}^{m-1}(2k)!\binom{n+k}{n-k}}{m!
\displaystyle\prod_{k=1}^{m-1}(2k)!\binom{m+k}{m-k}}  \nonumber \\
[0.2cm] & = &
(-1)^{\tfrac{m(m-1)}2}\binom{n}{m}\prod_{k=1}^{m-1}\dfrac{\binom{n+k}{n-k}}{\binom{m+k}{m-k}}
\, .  \label{Theorem.Hankel.minors.exact.proof.10}
\end{eqnarray*}
Substituting this expression for $T_m$ into~\eqref{Theorem.Hankel.minors.exact.proof.1}, we obtain~\eqref{Hankel.determinants.3},
%
%
as required.
\end{proof}
\begin{remark}
The determinant~\eqref{Theorem.Hankel.minors.exact.proof.9} is
connected to the Pad\'e approximations to the exponential, and its
value can be found, e.g., in~\cite[Section 1.2, p.9]{BG-M1980}.
\end{remark}

\begin{remark}\label{Remark:Pascal}
    On summing up the matrix transformations behind Theorem~\ref{Theorem.Hankel.minors.exact},
    one additionally arrives at the following factorisation of the matrix
    from~\eqref{Theorem.Hankel.minors.exact.proof.2}:
    \[
        \left[\binom{n}{i+j-1}\right]_{i,j=1}^m
        = \diag\left[(n-j)_{j+1}\right]_{j=0}^{m-1}
        \cdot \left[\frac{1}{(i+j-1)!}\right]_{i,j=1}^m
        \cdot\diag\left[(n)_{j}\right]_{j=0}^{m-1}
        \cdot \left[(-1)^{i+j}\binom{j}{i}\right]_{i,j=0}^{m-1}
        \, ,
    \]
    where
    \[
        \left[(-1)^{i+j}\binom{j}{i}\right]_{i,j=0}^{m-1}
        =\left(\left[\binom{j}{i}\right]_{i,j=0}^{m-1}\right)^{-1}
    \]
    is the inverse of the upper-triangular Pascal matrix. This inverse of the Pascal matrix is
    a product of $m-1$ bidiagonal terms stemming from column operations and application of the
    Pascal identity starting~\eqref{Theorem.Hankel.minors.exact.proof.4}. The Pascal matrix
    itself also has an analogous bidiagonal factorisation (but the terms are unsigned). The
    remaining term, whose determinant appears on the left-hand side of~\eqref{Theorem.Hankel.minors.exact.proof.9},
    can be expressed from the
    matrix~$\left[\binom{m}{i+j-1}\right]_{i,j=1}^m$ by the reverse transformations with~$n$ set
    to be~$m$, that is
    \[
        \left[\frac{1}{(i+j-1)!}\right]_{i,j=1}^m
        \!\!\!\!\!\!
        = \diag\left[\frac1{(m-j)_{j+1}}\right]_{j=0}^{m-1}
        \cdot \left[\binom{m}{i+j-1}\right]_{i,j=1}^m
        \!\!\!\!\!\!
        \cdot \diag\left[\frac1{(m)_{j}}\right]_{j=0}^{m-1}
        \cdot \left[\binom{j}{i}\right]_{i,j=0}^{m-1}
        \!\!\!.
    \]
    At the same time, the Hankel matrix~$\left[\binom{m}{i+j-1}\right]_{i,j=1}^m$ is built from
    the coefficients of the polynomial~$(x+1)^m$, and hence it may be further factorised
    into~$m\times m$ factors
    \[
        \left[\binom{m}{i+j-1}\right]_{i,j=1}^m =
        \begin{bmatrix}
            0&\hdots &0&0& 1\\[-3pt]
            \vdots&\iddots &0&1& 0\\[-3pt]
            \vdots&\iddots &1&0& 0\\[-3pt]
            0&\iddots &\iddots&\iddots& \vdots\\
            1&0&0&\hdots &0\\
        \end{bmatrix}
        \cdot
        \begin{bmatrix}
            1&0&0&\hdots &0\\[-3pt]
            1&1&0&\ddots &\vdots\\[-3pt]
            0&1&1&\ddots &0\\[-3pt]
            \vdots&\ddots&\ddots&\ddots& 0\\
            0&\hdots&0 &1&1\\
        \end{bmatrix}^m
        \!\!\!.
    \]
\end{remark}\medskip

Now~\eqref{coeff.b_k} and~\eqref{Hankel.determinants.3} give us the
exact formula for the entries of the matrix~\eqref{main.matrix}
whose characteristic polynomial is~\eqref{main.poly}.
\begin{corollary}\label{Corol.b_k}
    The coefficients $b_m$ of the relations~\eqref{three-term.rec.rel.general} have the form
\begin{equation}\label{coeff.k.2}
b_m=a^2\,\dfrac{n^2-m^2}{4m^2-1}\, , \quad \mbox{for}\; m=1,2,\ldots,n-1.
\end{equation}
\end{corollary}
\begin{proof}
Let us consider first the relation
\begin{eqnarray}
\dfrac{D_{m+1}(R)}{D_m(R)} & = &
\dfrac{(-1)^{\tfrac{(m+1)(m+2)}2}2^{(m+1)m}a^{(m+1)^2}\binom{n}{m+1}\displaystyle\prod_{k=1}^{m}\dfrac{\binom{n+k}{n-k}}
{\binom{m+1+k}{m+1-k}}}
{(-1)^{\tfrac{m(m+1)}2}2^{m(m-1)}a^{m^2}\binom{n}{m}\displaystyle\prod_{k=1}^{m-1}\dfrac{\binom{n+k}{n-k}}{\binom{m+k}{m-k}}}
\nonumber\\  [3pt]
& = & \displaystyle(-1)^{m+1}4^ma^{2m+1}\dfrac{n-m}{m+1}
\dfrac{\binom{n+m}{n-m}}{2m+1}\prod_{k=1}^{m-1}\dfrac{\binom{m+k}{m-k}}{\binom{m+1+k}{m+1-k}}\nonumber \\ [3pt]
& = &
(-1)^{m+1}4^ma^{2m+1}\dfrac{n-m}{m+1}\dfrac{\binom{n+m}{n-m}}{2m+1}\prod_{k=1}^{m-1}\dfrac{m+1-k}{m+1+k}\nonumber
\\  [3pt]
& = &
(-1)^{m+1}4^ma^{2m+1}\dfrac{n-m}{(2m+1)!}\displaystyle\binom{n+m}{n-m}(m!)^2.
\label{Corol.b_k.proof.1}
\end{eqnarray}

Now, from~\eqref{coeff.b_k} and~\eqref{Corol.b_k.proof.1}, we have
\begin{eqnarray*}
b_m & = & -\dfrac{D_{m+1}(R)/D_m(R)}{D_m(R)/D_{m-1}(R)}
\\ [.2cm]
& = &
-\dfrac{(-1)^{m+1}4^ma^{2m+1}\dfrac{n-m}{(2m+1)!}\displaystyle\binom{n+m}{n-m}(m!)^2}
{(-1)^{m}4^{m-1}a^{2m-1}\dfrac{n-m+1}{(2m-1)!}\displaystyle\binom{n+m-1}{n-m+1}((m-1)!)^2}\\
[.3cm]
& = &  4m^2a^2\,\dfrac{n-m}{2m(2m+1)}\dfrac{(n+m)!}{(n-m)!(2m)!}\dfrac{(n-m)!(2m-2)!}{(n+m-1)!}\\
[.3cm] & = & a^2\,\dfrac{n-m}{(2m-1)(2m+1)}\dfrac{(n+m)!}{(n+m-1)!}
=a^2\,\dfrac{n^2-m^2}{4m^2-1}\, .
\end{eqnarray*}
\end{proof}

As a conclusion, up to similarity, the only Schwarz matrix $J_n$ whose characteristic polynomial is $(z-a)^{n}$ is
\begin{equation}\label{binomial.matrix}
J_{n}=
\begin{bmatrix}
    an & 1 &&&\\
    -a^2\dfrac{n^2-1}{3} & 0 &1 && \\
    &-a^2\dfrac{n^2-4}{15} & \ddots & \ddots & \\
    && \ddots & \ddots & 1\\
    &&& -a^2\dfrac{1}{2n-3} & 0\\
\end{bmatrix}.
\end{equation}

Clearly, the spectrum of the matrix
\begin{equation*}
\tilde{J}_{n}= a
\begin{bmatrix}
n &  \dfrac{n-1}{1} &&&\\
-\, \dfrac{n+1}{3} & 0 & \dfrac{n-2}{3} && \\
&-\,\dfrac{n+2}{5} & \ddots & \ddots & \\
&& \ddots & \ddots & \dfrac{1}{2n-3}\\
&&& -1 & 0\\
\end{bmatrix}
\end{equation*}
is reduced to $a$ as well.

\section{Orthogonal polynomials}\label{Section:orthogonal.polynomials}

In this section, we consider the sequence $\{P_k(z)\}_{k=0}^{n-1}$ of the
characteristic polynomials of the leading principal minors of the
matrix~\eqref{binomial.matrix} satisfying the three-term recurrence
relation~\eqref{main.matrix}--\eqref{three-term.rec.rel.general}:
\begin{equation*}
P_{0}(z)=1,\quad P_1(z)=z-an,
\end{equation*}
\begin{equation}\label{orthogonal.polynomials}
P_{k+1}(z)=zP_{k}(z)+\dfrac{a^2(n^2-k^2)}{(2k-1)(2k+1)}P_{k-1}(z)\, , \quad \mbox{for}\;
k=1,\ldots,n-1.
\end{equation}
We give their explicit representations and study their even and odd parts. Then we show that $\{P_k(z)\}_{k=0}^{n-1}$ are orthogonal w.r.t.\ a linear functional supported at the only point~$z=a$. We also find the inverse Fourier transform of the corresponding distributional weight and depict the zeroes of the polynomials $P_k(z)$ for various $a$ and relations between~$k$ and~$n$ as $n\to+\infty$.

\subsection{Basic properties}\label{sect:P_k.basic.properties}
For the current subsection we assume~$a=1$ keeping in mind that the general case
$a\ne0$ may be obtained by
\begin{equation}\label{P.with.different.a}
    a^{-k}P_k(az) = P_k(z)\big|_{a=1},
\end{equation}
which immediately follows from~\eqref{orthogonal.polynomials}. Observe that
the monic Jacobi polynomials
\[
  \widehat P_{k}^{(-n,n)}(z)
  = \frac{2^k}{\binom{2k}{k}} P_{k}^{(-n,n)}(z)
  \qquad\left(
  = \frac{(-2)^k}{\binom{2k}{k}} P_{k}^{(n,-n)}(-z)
  \right)
\]
satisfy for~$k\ge 1$ the three-term recurrence relation~\cite[p.~153, eq.~(2.29)]{C1978}
\[
\widehat P_{k+1}^{(-n,n)}(z)
=
z\widehat P_{k}^{(-n,n)}(z)
-
\frac{4k(k-n)(k+n)k}{(2k-1)(2k+1)(2k)^2}
\widehat P_{k-1}^{(-n,n)}(z),
\]
which on reduction of similar terms turns into \eqref{orthogonal.polynomials} with~$a=1$. Similarly,
$
\widehat P_{0}^{(-n,n)}(z)\equiv 1$
and
$\widehat P_{1}^{(-n,n)}(z)=z-n$.
Therefore, our~$P_k(z)$ are exactly~$\widehat P_{k}^{(-n,n)}(z)$, namely,
for all $k=0,\dots,n,$
\begin{equation}\label{eq:P.explicit.2F1}
P_k(z)=\widehat P_k^{(-n,n)}(z)
=(-2)^{k}\frac{(n+1)_k}{(k+1)_k}\,
{}_2F_1\left(\genfrac{.}{|}{0pt}{0}{-k,\ k+1}{n+1}\dfrac{1+z}2\right).
\end{equation}
Also,
$$
P_k(z)\big|_{a=1}=\dfrac{(-1)^k}{(k+1)_k}\left(\dfrac{1-z}{1+z}\right)^n
\dfrac{d^k}{dz^k}\left[(1-z)^{n-k}(1+z)^{n+k}\right],
$$
$$
(1-z^2)P_k''(z)+2(n-z)P_k'(z)+k(k+1)P_k(z)=0.
$$

Although the three-term recurrence degenerates to a two-term recurrence
for~$k=n$, the sequence~$\{P_k\}_{k=0}^{n}$ may be further extended to an
infinite sequence. Properties of infinite sequences of the monic
polynomials~$\{\widehat P_k^{(-n,\beta)}\}_{k=0}^{\infty}$ with~$\beta\notin\{-1,-2,\dots\}$
were studied in~\cite{Alfaro_et_al.1999}, where orthogonality was introduced via a bilinear form.

Now, on introducing two auxiliary families of polynomials via
\begin{equation}\label{orthogonal.polynomials.aux1}
    f_0(z)=1\, ,\quad f_1(z)=z\, ,\quad\text{and}\quad
    f_{k+1}(z)=zf_{k}(z)+\dfrac{n^2-k^2}{4k^2-1}f_{k-1}(z)
\end{equation}
and
\begin{equation}\label{orthogonal.polynomials.aux2}
    g_0(z)=0\, ,\quad g_1(z)=1\, , \quad\text{and}\quad
    g_{k+1}(z)=zg_{k}(z)+\dfrac{n^2-k^2}{4k^2-1}g_{k-1}(z)
\end{equation}
with~$k=1,\dots,n-1$, one immediately obtains that the decomposition
\[
    P_k(z) = f_k(z) - n g_k(z)
\]
is valid for all~$k$. The formula~\eqref{orthogonal.polynomials.aux1} determines monic orthogonal polynomials supported on the imaginary axis, and~\eqref{orthogonal.polynomials.aux2} gives exactly the corresponding monic numerator polynomials.
In particular, for each~$k$ the zeroes of~$f_k$ and~$g_k$ are all simple and pure imaginary, they also interlace on the imaginary axis, see \cite[p.~86]{C1978}.
Therefore, the Hermite-Biehler theorem implies that all zeroes of the polynomial~$P_k(z)=f_k(z) - n g_k(z)$ have positive real parts,~\cite{KreinNaimark.1937} (see also~\cite{HT2012} and references there). That is to say, all zeroes of the polynomial~$P_k$ for arbitrary~$a\ne 0$ lie in the half-plane $\{z\in\mathbb C\, :\, \operatorname{Re}\left(z/a\right)>0\}$, which is clearly seen on the figures in Sections \ref{section:Bessel} and \ref{section:Zeroes}.

In fact, the polynomials $f_k$ and~$g_k$ may also be expressed via classical families of orthogonal polynomials.
The three-term relations~\eqref{orthogonal.polynomials.aux1}
and~\eqref{orthogonal.polynomials.aux2} are known to determine monic versions of the so-called
associated Jacobi polynomials~$\widehat P_k^{(\alpha,\beta)}(z;c)$, namely,
\[
    f_k(z)=\widehat P_k^{(-n,-n)}(z;n)\quad\text{and}\quad
    g_{k}(z)=\widehat P_{k-1}^{(-n,-n)}(z;n+1).
\]
The explicit representation of the associated Jacobi polynomials is due to
Wimp~\cite[p.~182]{Ismail}, however our specific case allows to write these polynomials using the simple expressions
\[
    f_k(z) 
    =
    \frac{(-1)^k P_k(-z)+P_k(z)}{2}
\quad\text{and}\quad
    g_k(z) 
    =
    \frac{(-1)^kP_k(-z)- P_k(z)}{2n}.
\]
Note that the sequence~$\{f_k\}_{k=0}^{n-1}$ consists of the so-called (monic) exceptional Jacobi polynomials~\cite[p.~97]{Ismail}, although the parameter~$\alpha=-n$ is non-standard for them.
Observe also that $f_k$ and $-ng_k$ turn into $Q$ and $q$ from~\eqref{even.odd.parts} on letting~$k=n$.

\subsection{Orthogonality}\label{subsection:orthogonality}

Now, let us introduce the following linear functional
$\mathcal{L}_{n}^{a}(f(x))$ acting on $(n-1)$th differentiable
functions $f$ and depending on the positive integer number~$n$ and
the (complex) nonzero number $a$:
\[
\mathcal{L}_{n}^{a}(f(x)) = \sum_{k=1}^{n}\dfrac{f^{(k-1)}(a)\alpha_{k}}{(k-1)!} \label{linear.functional}
=\sum_{k=1}^{n}\dfrac{f^{(k-1)}(a)}{(k-1)!}\dfrac{q^{(n-k)}(a)}{(n-k)!}
= \dfrac1{(n-1)!}\dfrac{d^{n-1}[f(x)q(x)]}{dx^{n-1}}\Big|_{x=a} \, .
\]
%
%
where the numbers $\alpha_k$ are defined in~\eqref{alpha_k}.

Alternatively, assuming that $f$ is analytic near the point $a$ this functional may be expressed in a more usual form --- as an integral with a continuous (complex-valued) Jacobi weight $(z-a)^{-n}(z+a)^n$ over the circle $|z-a|=2|a|$, which is typical for the Jacobi polynomials with such parameters, cf.\ \cite[p.~222]{Geronimus1940}.
The weight is analytic in~$\mathbb C\setminus\{0\}$, so the circle may be replaced with any Jordan curve enclosing the point $z=a$.
Due to Cauchy's integral formula, this integral will essentially express the same functional $\mathcal{L}_{n}^{a}$, namely:
\[
\begin{split}
\mathcal{L}_{n}^{a}(f)&= \dfrac1{(n-1)!}\dfrac{d^{n-1}[f(z)q(z)]}{dz^{n-1}}\Big|_{z=a}
=
\frac 1{2\pi i}
\oint\limits_{|z-a|=2|a|} \frac{f(z)q(z)}{(z-a)^n} dz
\\
&=
\frac 1{2\pi i}
\oint\limits_{|z-a|=2|a|} f(z) \cdot\frac 1{2}\left(
\left(\frac{z-a}{z-a}\right)^n -\left(\frac{z+a}{z-a}\right)^n\right) dz
\\
&=
\frac {i}{4\pi}\oint\limits_{|z-a|=2|a|} f(z) \left(\frac{z+a}{z-a}\right)^n dz
.
\end{split}
\]

With this functional we can represent the rational function $R(z)$
defined in~\eqref{rat.func.general} as follows
\begin{equation}\label{Cauchy.integral}
R(z)=\mathcal{L}_{n}^{a}\left(\dfrac1{z-x}\right)=\sum_{k=1}^{n}\dfrac{\alpha_{k}}{(z-a)^{k}}\, .
\end{equation}
Moreover, from~\eqref{linear.functional} we have
\begin{equation*}\label{moments.functional}
\mathcal{L}_{n}^{a}(x^m)=s_m\, , \quad \mbox{for}\; m=0,1,\ldots,
\end{equation*}
where $s_m$ are the moments defined in~\eqref{s_k}. In fact,
\begin{eqnarray*}\label{moments.functional.proof}
\mathcal{L}_{n}^{a}(x^m) & = & \sum_{k=1}^{n}\dfrac{(x^m)^{(k-1)}\Big|_{x=a}}{(k-1)!}\, \alpha_{k} 
=a^m\alpha_1+\sum_{k=2}^{n}\dfrac{m(m-1)\cdots(m-k+2)a^{m-k+1}}{(k-1)!}\, \alpha_{k}\\
\\
& = & \sum_{k=1}^{n}\binom{m}{k-1}a^{m-k+1}\alpha_{k}=s_m\, .
\end{eqnarray*}


We now prove that the polynomials $P_k(z)$, for~$k=0,1,\ldots,n-1$, defined in~\eqref{orthogonal.polynomials} are
orthogonal w.r.t. the functional~\eqref{linear.functional},
that is,
\begin{equation*}\label{orthogonality}
\mathcal{L}_{n}^{a}(P_m(x)P_k(x))=C_m\delta_{km} \ \ \quad \text{for}\quad
m,k=0,1,\ldots,n-1\, ,
\end{equation*}
where $\delta_{km}$ is the Kronecker symbol. To this end, we find the three-term recurrence relations for monic polynomials~$Q_m(z)$
orthogonal w.r.t. the functional $\mathcal{L}_n^a(f)$ and compare the obtained relations with~\eqref{orthogonal.polynomials}. Firstly,
let us note that all the minors $D_{m}(R)$, $m=1,\ldots,n$, defined in~\eqref{Hankel.determinants.1} are nonzero, i.e.,
\begin{equation*}
D_{m}(R)\neq0\, , \quad \mbox{for}\; m=1,\ldots,n,
\end{equation*}
by~\eqref{Hankel.determinants.3}, so such a sequence exists (see, e.g.,~\cite[Theorem 3.1]{C1978}). These orthogonal polynomials
have the form~\cite[p.~21]{Ismail}
\begin{equation}\label{orthopoly.functional}
Q_{m}(z)=\dfrac1{D_{m}(R)}\begin{vmatrix}
    s_0 &s_1 &s_2 &\dots &s_{m}\\
    s_1 &s_2 &s_3 &\dots &s_{m+1}\\
    \vdots&\vdots&\vdots&\ddots&\vdots\\
    s_{m-1} &s_{m} &s_{m+1} &\dots &s_{2m-1}\\
    1 &z &z^2 &\dots &z^{m}\\
\end{vmatrix}\, ,
\end{equation}
for $m=1,\ldots,n$. The roots of the last polynomial $Q_{n}(z)$ (partially)
define the functional $\mathcal{L}_{n}^{a}(f)$. If $\mathcal{L}_{n}^{a}(f)$ is
a positive or negative functional, $Q_{n}(z)$ would have exactly $n$
real distinct roots that are the growth point of the measure
generated by $\mathcal{L}_{n}^{a}(f)$, see, e.g.,~\cite[Section 2.2]{Ismail}. But the
formula~\eqref{Hankel.determinants.3} shows that
$\mathcal{L}_{n}^{a}(f)$ is neither a positive nor a negative
functional, since the minors $D_{m}(R)$ are not all positive nor of
alternating signs $(\mbox{sgn}\, D_{m}(R)\neq(-1)^{m}$, for
$m=1,\ldots,n$).

The polynomials $Q_m(z)$ satisfy a three-term recurrence relation
\begin{equation*}
Q_{m+1}(z)=(z-c_{m+1})Q_m(z)-d_{m+1}Q_{m-1}(z).
\end{equation*}
It is well known, e.g.,~\cite{C1978,HT2012,Ismail,W1948}, that
\begin{equation}\label{three-term.rec.rel.ortopoly.Q.coeff.b}
d_{m+1}=\dfrac{D_{m-1}(R)D_{m+1}(R)}{D^2_m(R)}=-b_m,
\end{equation}
so the polynomials $Q_m(z)$ satisfy the recurrence relation
\begin{equation}\label{three-term.rec.rel.ortopoly.Q}
Q_{m+1}(z)=(z-c_{m+1})Q_m(z)+b_mQ_{m-1}(z)\, , \quad \mbox{for}\; m=0,1,\cdots,n-1,
\end{equation}
where $Q_{-1}(z)=0$ and $Q_{0}(z)=1$.

Furthermore, if $$Q_m(z)=z^m+\displaystyle\sum_{i=1}^{m}A_{im}z^{m-i}\,
,$$ then (cf., e.g.,~\cite{C1978})
\begin{equation}\label{three-term.rec.rel.coeff.c}
c_1=A_{11}\, , \quad c_{m+1}=A_{1m}-A_{1,m+1}\, , \quad \mbox{for}\; m=1,\ldots,n-1\, ,
\end{equation}
and, from~\eqref{orthopoly.functional}, we have
\begin{equation}\label{orthopoly.functional.2nd.coef}
A_{1,m}=-\dfrac{D'_{m}(R)}{D_{m}(R)}\, , \quad \mbox{for}\; m=1,\ldots,n.
\end{equation}
where
\begin{equation*}\label{shifted.Hankel.determinant}
D'_{m}(R)=
\begin{vmatrix}
    s_0 &s_1 &s_2 &\cdots &s_{m-2}&s_{m}\\
    s_1 &s_2 &s_3 &\cdots &s_{m-1}&s_{m+1}\\
    \vdots&\vdots&\vdots& \cdots &\vdots&\vdots\\
    s_{m-1} &s_{m} &s_{m+1} & \cdots & s_{2m-3} &s_{2m-1}\\
\end{vmatrix}\, , \quad \mbox{for}\; m=1,\ldots,n\, .
\end{equation*}

To find the formula for the minors $D'_{m}(R)$, we need the
following lemma which can be proved in the same way as Lemma~\ref{Lemma.Hankel.minors.equivalency}.

\begin{lemma}
The minors $D'_{m}(R)$ can be represented in the following form:
\begin{eqnarray}\label{Hankel.determinants.222}
D'_{m}(R) & = &\begin{vmatrix}
    \alpha_1 &\alpha_2 &\alpha_3 &\cdots &\alpha_m\\
    \alpha_2 &\alpha_3 &\alpha_4 &\cdots &\alpha_{m+1}\\
    \vdots&\vdots&\vdots&\cdots&\vdots\\
    \alpha_{m-1} &\alpha_{m} &\alpha_{m+1} &\cdots &\alpha_{2m-2}\\
    \alpha_{m+1} &\alpha_{m+2} &\alpha_{m+3} &\cdots &\alpha_{2m}
\end{vmatrix} \\ \nonumber
& & \; +\, am
\begin{vmatrix}
    \alpha_1 &\alpha_2 &\alpha_3 &\cdots &\alpha_m\\
    \alpha_2 &\alpha_3 &\alpha_4 &\cdots &\alpha_{m+1}\\
    \vdots&\vdots& \vdots&\cdots&\vdots\\
    \alpha_{m-1} &\alpha_{m} &\alpha_{m+1} &\cdots &\alpha_{2m-2}\\
    \alpha_{m} &\alpha_{m+1} &\alpha_{m+2} &\cdots &\alpha_{2m-1}
\end{vmatrix},
\end{eqnarray}
for $m=1,\dots,n$, and setting $\alpha_j=0$, for $j>n$.
\end{lemma}

Furthermore, by the same technique that we used in proof of Theorem~\ref{Theorem.Hankel.minors.exact},
we may conclude the next theorem (here, we omit the proof to avoid long and unnecessary repetitions).

\begin{theorem}\label{Thm.sec.3}
The following identity holds for $m=1,\dots,n$,
\begin{equation}\label{Hankel.determinants.2222}
\begin{vmatrix}
    \alpha_1 &\alpha_2 &\alpha_3 &\cdots &\alpha_m\\
    \alpha_2 &\alpha_3 &\alpha_4 &\cdots &\alpha_{m+1}\\
    \vdots&\vdots&\vdots&\cdots&\vdots\\
    \alpha_{m-1} &\alpha_{m} &\alpha_{m+1} &\dots &\alpha_{2m-2}\\
    \alpha_{m+1} &\alpha_{m+2} &\alpha_{m+3} &\dots &\alpha_{2m}
\end{vmatrix}=
a(n-m)
\begin{vmatrix}
    \alpha_1 &\alpha_2 &\alpha_3 &\cdots &\alpha_m\\
    \alpha_2 &\alpha_3 &\alpha_4 &\cdots &\alpha_{m+1}\\
    \vdots&\vdots&\vdots&\cdots&\vdots\\
    \alpha_{m-1} &\alpha_{m} &\alpha_{m+1} &\cdots &\alpha_{2m-2}\\
    \alpha_{m} &\alpha_{m+1} &\alpha_{m+2} &\cdots &\alpha_{2m-1}
\end{vmatrix},
\end{equation}
where we set $\alpha_j=0$, for $j>n$, and for $m=1$ we mean that the left determinant equals~$\alpha_2$.
\end{theorem}

So, from~\eqref{Hankel.determinants.222} and~\eqref{Hankel.determinants.2222} we have
\begin{equation*}
D'_{m}(R)=an
\begin{vmatrix}
    \alpha_1 &\alpha_2 &\alpha_3 &\dots &\alpha_m\\
    \alpha_2 &\alpha_3 &\alpha_4 &\dots &\alpha_{m+1}\\
    \vdots&\vdots&\vdots&\cdots&\vdots\\
    \alpha_{m-1} &\alpha_{m} &\alpha_{m+1} &\dots &\alpha_{2m-2}\\
    \alpha_{m} &\alpha_{m+1} &\alpha_{m+2} &\dots &\alpha_{2m-1}
\end{vmatrix},
\end{equation*}
i.e., $$D'_{m}(R)=an D_{m}(R)\, .$$

Now,~\eqref{Hankel.determinants.2} and~\eqref{orthopoly.functional.2nd.coef} imply
\begin{equation*}
A_{1,m}=-an\, , \quad \mbox{for}\; m=1,2,\ldots,n\, .
\end{equation*}
Hence
\begin{equation}\label{three-term.rec.rel.ortopoly.Q.coeff.c}
c_1=an \quad \mbox{and} \quad c_{m}=0\, , \quad \mbox{for}\; m=2,\ldots,n\, .
\end{equation}
by~\eqref{three-term.rec.rel.coeff.c}.

Thus, from~\eqref{three-term.rec.rel.ortopoly.Q.coeff.b}--\eqref{three-term.rec.rel.ortopoly.Q}
and~\eqref{three-term.rec.rel.ortopoly.Q.coeff.c} we have
\begin{equation*}\label{orthogonal.polynomials2}
Q_{m+1}(z)=zQ_{m}(z)+\dfrac{a^2(n^2-m^2)}{(2m-1)(2m+1)}Q_{m-1}(z)\, , \quad \mbox{with}\;
m=1,\ldots,n\, ,
\end{equation*}
with initial conditions $Q_{0}(z)=1$, and $Q_1(z)=z-an$. Therefore,
the monic polynomials~$P_k(z)$ coincide with the monic polynomials~$Q_m(z)$
which are orthogonal with respect to the functional $\mathcal{L}_{n}^{a}(f(x))$ defined
in~\eqref{linear.functional} and
\begin{eqnarray*}
C_m  &= &  \mathcal{L}_{a}^{n}(P^2_m(x)) 
=b_0\prod_{k=1}^{m}(-b_k)
=(-1)^{m+1}a^{2m+1}n\prod_{k=1}^{m}\dfrac{(n^2-k^2)}{4k^2-1}\\
& = & (-1)^{m+1}a^{2m+1}\binom{n+m}{2m+1}\,\dfrac{(2m)!!}{(2m-1)!!}\, ,
\end{eqnarray*}
%
or, equivalently,
\begin{equation*}\label{norm.2}
C_m=(-1)^{m+1}a^{2m+1}\sqrt{\pi}\binom{n+m}{2m+1}\,
\dfrac{\Gamma(m+1)}{\Gamma\left(m+\tfrac12\right)},
\end{equation*}
for $m=0,1,\ldots,n-1$.

For example, setting $m=n-1$, we have the following interesting fact:
\begin{eqnarray*}
C_{n-1} & = & \mathcal{L}_{a}^{n}(P^2_{n-1}(x))
(-1)^{n}a^{2n-1}\sqrt{\pi}\,
\dfrac{\Gamma(n)}{\Gamma\left(n-\tfrac12\right)}\sim(-1)^{n}a^{2n-1}\sqrt{\pi
n}\quad\text{as}\quad n\to+\infty.
\end{eqnarray*}

Having the moments $s_m$ defined in~\eqref{s_k} which are related to the
orthogonal polynomials~$\{P_m(z)\}_{m=0}^{n-1}$, one can introduce the distributional
weight function $w_a$ characterised by those moments (see~\cite{Morton.Krall.1978})
such that
\begin{equation}\label{functional.w}
\langle w_a,f\rangle=\sum\limits_{m=0}^{\infty}\dfrac{s_mf^{(m)}(0)}{m!}
\end{equation}
for infinitely differentiable function $f$ with a compact support. Therefore, the distributional
weight function $\omega_a$ of this functional has the form~\cite{Morton.Krall.1978}
\begin{equation}\label{distrib.weight}
\omega_a(x)=\sum\limits_{m=0}^{\infty}\dfrac{(-1)^ms_m\delta^{(m)}(x)}{m!}\, ,
\end{equation}
where $\delta(x)=\delta^{(0)}(x)$ is the Dirac delta-function, and $\delta^{(m)}(x)$, for $m=1,2,\ldots$,
are its derivatives. From~\eqref{s_k} it easily follows that the moments $s_m$ satisfy the inequality
\begin{equation*}
|s_m|\leqslant C\,|a|^m\, m!\, ,
\end{equation*}
for certain constants $C=C(a,n)>0$, and the distributional weight $w_a(x)$ has the inverse Fourier transform~\cite{Morton.Krall.1978}
\begin{equation*}
F^{-1}w_a(t)=\dfrac1{2\pi}\int\limits_{-\infty}^{+\infty}e^{-ixt}\omega_a(x)dx=\dfrac1{2\pi}\sum\limits_{m=0}^{+\infty}\dfrac{s_m(-it)^m}{m!}\, .
\end{equation*}
The sum of this series can be found explicitly.
\begin{theorem}
The inverse Fourier transform of the distributional weight~\eqref{distrib.weight} corresponding to the functional $w_{a}$
has the form
$$
F^{-1}w_a(t)=\dfrac{-ae^{-iat}}{2\pi}\,L^{(1)}_{n-1}(2iat),\quad t\in\mathbb{C}\, ,
$$
where $L^{(\alpha)}_{n}(x)$ denotes the (generalised) Laguerre polynomial.
\end{theorem}
\begin{proof}
From~\eqref{s_k} one has
\begin{eqnarray*}
\displaystyle F^{-1}w_a(t) & = & \dfrac{-a}{2\pi}\sum\limits_{m=0}^{\infty} \sum\limits_{k=1}^{n}2^{k-1}a^{m}\binom{m}{k-1}\binom{n}{k}\dfrac{(-it)^m}{m!}\\
& =  & \dfrac{-a}{2\pi}\sum\limits_{k=1}^{n}\dfrac{2^{k-1}}{(k-1)!} \binom{n}{k}\sum\limits_{m=0}^{\infty}\dfrac{(-ait)^m}{(m-k+1)!}
\\ & = &
\dfrac{-ae^{-ait}}{2\pi}\sum\limits_{k=0}^{n-1}\dfrac{1}{k!}\binom{n}{k+1}(-2ait)^{k}
=\dfrac{-ae^{-iat}}{2\pi}\,L^{(1)}_{n-1}(2iat)\, ,
\end{eqnarray*}
where the series converges for $t\in\mathbb{C}$.
\end{proof}

In~\eqref{functional.w} we defined the functional $w_a$ for infinitely differentiable functions with compact support.
In fact it will be extended to a larger class of functions. Indeed, let us find the functional $w_a$ by applying
the Fourier transform to $F^{-1}w_a(t)$.
\begin{equation*}
\begin{split}
F\left(F^{-1}w_a(t)\right)
& =  -\dfrac{a}{2\pi}F\left(e^{-iat}\,L^{(1)}_{n-1}(2iat)\right)\\[3pt]
&=  \sum\limits_{k=1}^{n}\dfrac{-2^{k-1}a^k\binom{n}{k}}{(k-1)!} \,\dfrac{(-i)^{k-1}F(e^{-iat}\,t^{k-1})}{2\pi} 
=\sum\limits_{k=1}^{n}\dfrac{\alpha_k\delta^{(k-1)}(x-a)}{(k-1)!}\, ,
\end{split}
\end{equation*}
where $\alpha_k$ are defined in~\eqref{alpha_k}. Thus, we can continuously extend the functional $w_a$ to the set of all $(n-1)$-times differentiable functions.
From~\eqref{linear.functional} it follows that the continuously extended functional $w_a$ coincides with $\mathcal{L}_a^n$.

Note that since the matrix $J_n$ defined in~\eqref{main.matrix} is tridiagonal, it is non-derogatory. Thus, the normalized (by the first unit entry)
eigenvector corresponding to the unique eigenvalue of $J_n$ is $\mathbf{u}_0:=\mathbf{u}(a)$ where $\mathbf{u}(z)=(P_0(z),P_1(z),\ldots,P_{n-2}(z),P_{n-1}(z))^T$.
The corresponding generalised eigenvectors are $\mathbf{u}_k:=\dfrac{d^k\mathbf{u}(z)}{dz^k}\Bigg|_{z=a}$, $k=1,\ldots,n-1$.

\subsection{Relation to Bessel polynomials}\label{section:Bessel}

The Bessel polynomials $\{B_n(x)\}_{n=0}^{\infty}$ is a sequence of polynomials satisfying the differential
equation
\begin{equation*}
x^2\dfrac{d^2y}{dx^2}+(2x+2)\dfrac{dy}{dx}=n(n+1)y.
\end{equation*}
They occur naturally in the theory of traveling spherical waves~\cite{KrallFrink.1948}. These (monic) polynomials
satisfy the following three-term recurrence relations
\begin{equation}\label{orthogonal.polynomials.Bessel}
\begin{array}{l}
B_{0}(z)=1,\quad B_1(z)=z+1,\\
\\
B_{k+1}(z)=zB_{k}(z)+\dfrac{1}{4k^2-1}B_{k-1}(z)\, , \quad \mbox{for}\;
k=1,2,\ldots
\end{array}
\end{equation}
The Bessel polynomials were introduced by H.L. Krall and O. Frink in~\cite{KrallFrink.1948}. They satisfy complex orthogonality conditions
\begin{equation*}
\int\limits_{C}B_{n}(z)B_{m}(z)e^{-2/z}dz=0,\qquad\quad m\neq n,
\end{equation*}
where $C$ is a closed path in the complex plane enclosing the origin. Their orthogonality on the real axis
was a subject of discussion for many years until this problem was resolved in a few different ways,
see~\cite{Duran,Evans_et_al.1993,Morton.Krall.1978}. We note that on the real line, the support of
the functional of orthogonality for the Bessel polynomials consists of a unique point $z=0$ which is an essential
singularity. Moreover, from the recurrence relations~\eqref{orthogonal.polynomials.Bessel}, it follows that every
monic Bessel polynomial is the characteristic polynomial of a Schwarz matrix defined by these relations and, therefore,
is stable. Since the polynomials $\{P_k\}_{k=0}^{n-1}$ satisfying~\eqref{orthogonal.polynomials}
and orthogonal w.r.t. the functional $\mathcal{L}_a^n$ defined in~\eqref{linear.functional} are
characteristic polynomials of Schwarz matrices stable for $a<0$, and the support of the functional $\mathcal{L}_a^n$
consists of one point, we may say that $\{P_k\}_{k=0}^{n-1}$ can be treated as a finite discrete analogue of Bessel
polynomials. Additionally, if one sets $a=-\dfrac1n$, then the three-term recurrence relations~\eqref{orthogonal.polynomials}
for the polynomials $P_k(z)$ take the form
\begin{equation*}
\begin{array}{l}
P_{0}(z)=1,\quad P_1(z)=z+1,\\[5pt]
P_{k+1}(z)=zP_{k}(z)+\dfrac{P_{k-1}(z)}{4k^2-1}-\dfrac1{n^2}\,\dfrac{k^2}{4k^2-1}P_{k-1}(z)\, , \quad \mbox{for}\;
k=1,\ldots,n-1,
\end{array}
\end{equation*}
from which it follows that as $n\to+\infty$, these recurrence relations are close to~\eqref{orthogonal.polynomials.Bessel}.
Thus, the polynomials $P_k(z)$ approach $B_k(z)$ coefficient-wise, and, therefore,
on the compact subsets of the complex plane. So for $a=-\dfrac1{n}$, sufficiently large $n$, and relatively small~$k$ (comparing with $n$)
one can consider the polynomials $P_k(z)$ as an approximation of the monic Bessel polynomials $B_k(z)$.
Thus, we can expect that for a fixed $k$ and large $n$ the zeroes of $P_k(z)$ will approximate zeroes of $B_k(z)$, and
as $k\to+\infty$ and $n\to+\infty$ with $\dfrac{k}{n}\to0$ the zeroes of $P_k(z)$ will approach the same curve as the
zeroes of $B_k(z)$ for $k\to+\infty$, see~\cite{deBruin.Saff.Varga.1,deBruin.Saff.Varga.2} and references there.
Figures~\ref{Pic.50.Bes} and~\ref{Pic.75.Bes} show the zeroes of the polynomials $B_{50}(z)$ and $B_{75}(z)$,
respectively, and the zeroes of $P_{50}(z)$ and $P_{75}(z)$ approaching them as $n\to+\infty$.

\begin{figure}[h]
\begin{multicols}{2}
\hfill
\includegraphics[scale=0.40]{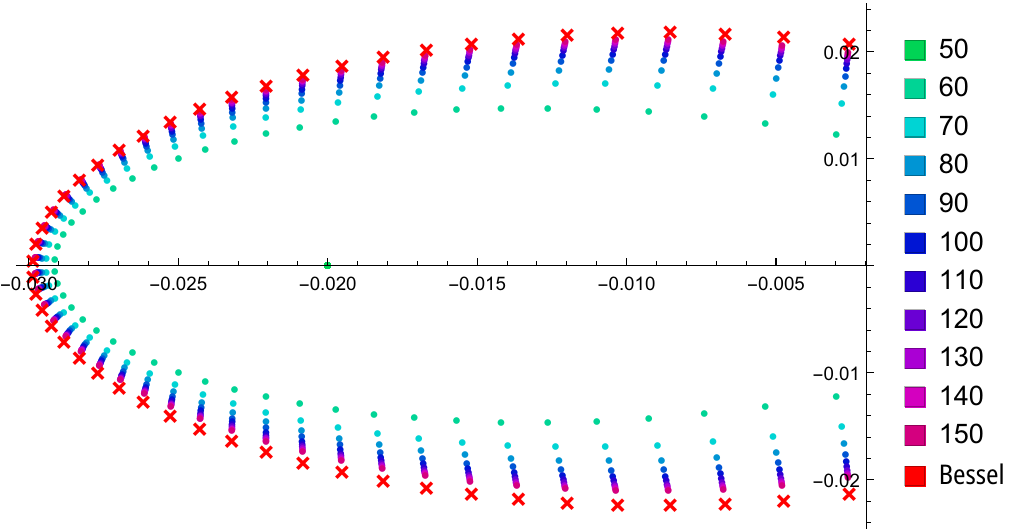}
\hfill\hfill
\caption{Zeroes of $B_{50}$ and $P_{50}$ for $n=50,\dots,150$}\label{Pic.50.Bes}
\hfill
\includegraphics[scale=0.40]{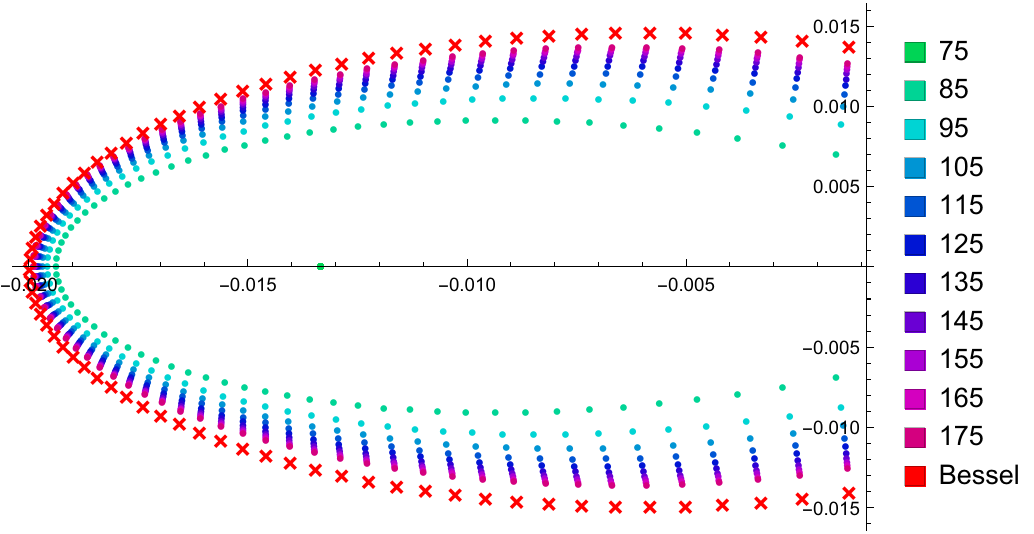}
\hfill\hfill
\caption{Zeroes of $B_{75}$ and $P_{75}$ for $n=75,\dots,175$}\label{Pic.75.Bes}
\end{multicols}
\end{figure}

Figures~\ref{Pic.k2.Bes} and~\ref{Pic.k2x.Bes} represent the zeroes of the polynomials
$P_{k}(z)$ and $B_{k}(z)$ for~$n=k^2$ and~$k=5,\ldots,24$. There is a visibly good approximation
of~$B_{k}(z)$ by~$P_k(z)$ in such a case.

\begin{figure}[h]
\begin{multicols}{2}
\hfill
\includegraphics[scale=0.40]{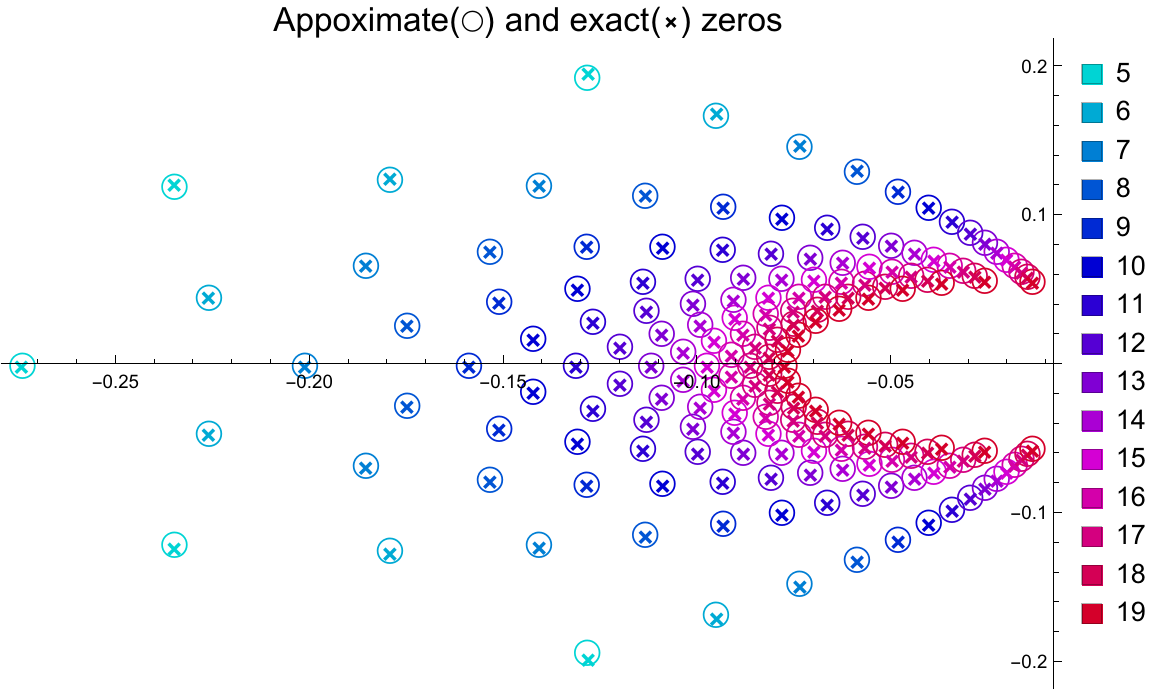}
\hfill\hfill\hfill\hfill
\caption{Zeroes of $P_{k}$ and $B_{k}$ for $n=k^2$ and $k=5,\ldots,19$}\label{Pic.k2.Bes}
\hfill
\includegraphics[scale=0.40]{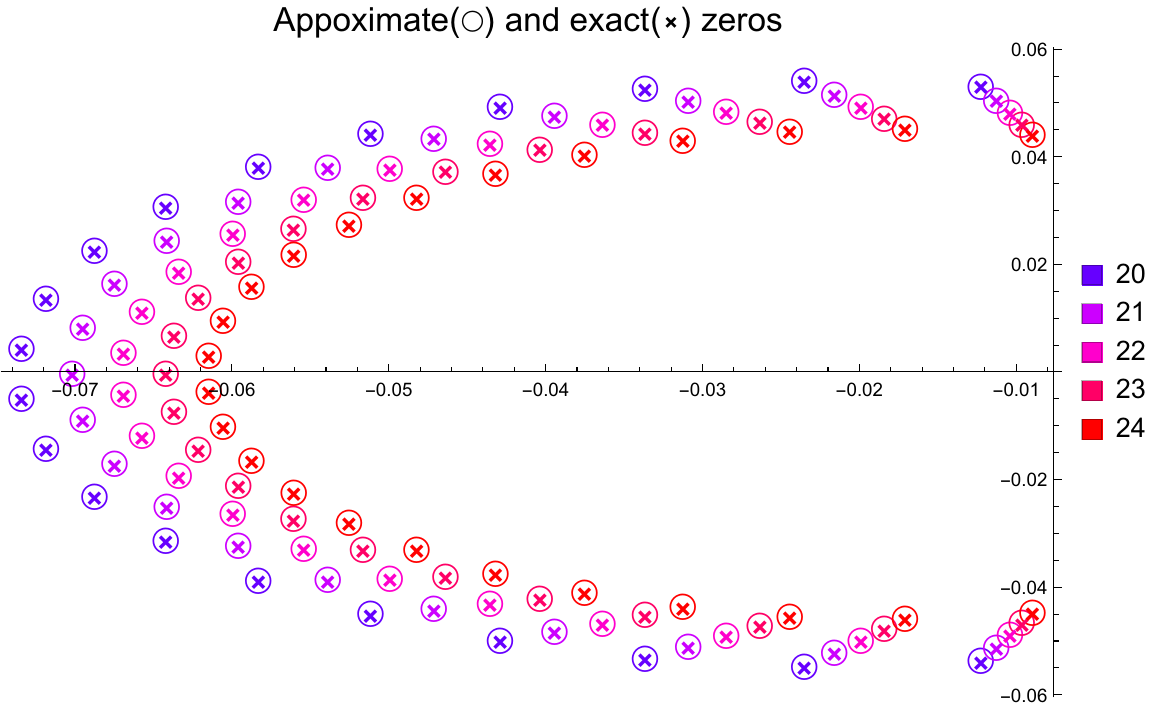}
\hfill\hfill
\caption{Zeroes of $P_{k}$ and $B_{k}$ for $n=k^2$ and $k=20,\ldots,24$}\label{Pic.k2x.Bes}
\end{multicols}
\end{figure}

Note also that if $a=-\dfrac1n$, then as $n\to+\infty$ the limits of the moments $s_m(n)$, $m=0,1,2,\ldots$,
are close to the moments of the Bessel polynomials (but not the same). At the same time,
the inverse Fourier transform of the distributional weight function
corresponding to the Bessel polynomials has the form
$$
F^{-1}w(t)=-\dfrac1{2\pi}\sum\limits_{k=0}^{+\infty}\dfrac{2^{k+1}(it)^k}{k!(k+1)!}=-\dfrac{I_1(\sqrt{8it})}{\pi\sqrt{8it}},
$$
where $I_1(x)$ is the modified Bessel function~\cite{Morton.Krall.1978}.

\subsection{Zeroes of polynomials $P_k(z)$ for fixed $a$}\label{section:Zeroes}

Apart of approximation of Bessel polynomials, it is interesting to study how the zeroes
the polynomials $P_k$ behave for a fixed $a\in\mathbb{C}\setminus\{0\}$.
As it was already observed in Section~\ref{sect:P_k.basic.properties},
the results on the polynomials~$P_k$ for a certain fixed~$a\in\mathbb{C}\setminus\{0\}$
immediately extend to other values of~$a$ due to
\[
a^{-k}P_k(az) = P_k(z)\big\vert_{a=1}
\quad\text{or, equivalently,}\quad
(-a)^{-k}P_k(-az) = P_k(z)\big\vert_{a=-1}.
\]
Accordingly, varying values of~$a$ only scales and/or rotates the zeroes of these polynomials in~$\mathbb{C}$, see Figure~\ref{Pic.acx}. We shall concentrate on the case~$a=-1$ which is related to the previous section.
\begin{figure}[h]
\begin{multicols}{2}
\hfill
\includegraphics[scale=0.46]{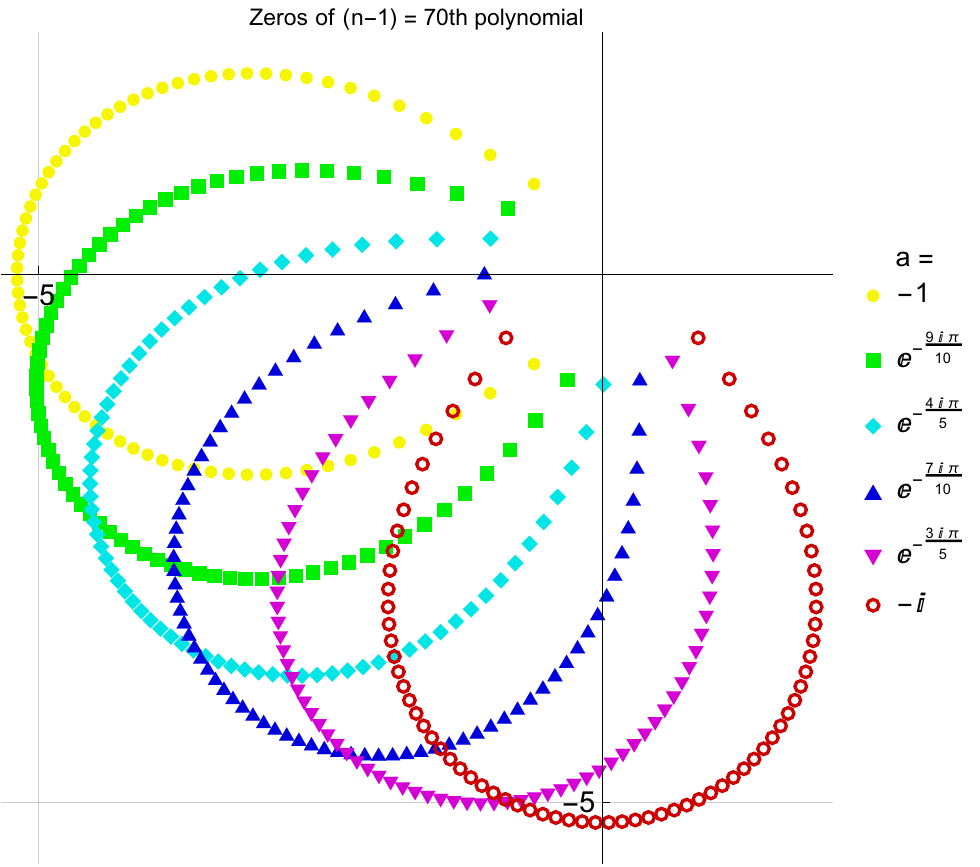}
\hfill
\caption{Zeroes of $P_{k}$ for $a=e^{i\varphi}$ with $n=k+1=71$, and $\varphi=-\pi,\dots,-\frac\pi2$}\label{Pic.acx}
\hfill
\includegraphics[scale=0.4]{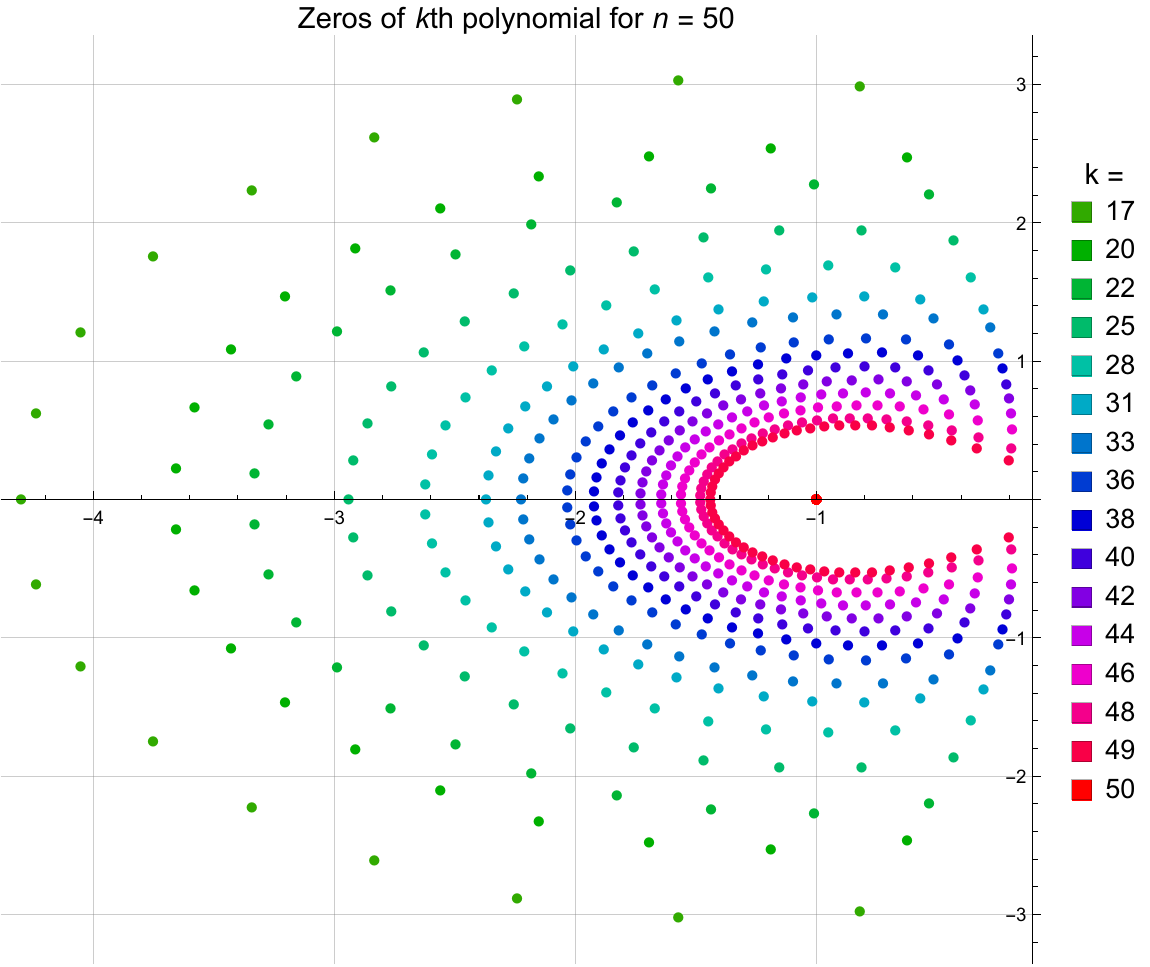}
\hfill\hfill
\caption{Zeroes of $P_{k}$ for $a=-1$, $k=17,\ldots,50$, and $n=50$}\label{Pic.ac}
\hfill
\end{multicols}
\end{figure}

Figures~\ref{Pic.a} and~\ref{Pic.ax}
represent the zeroes of the polynomials $P_{n/10}(z)$ and $P_{n-1}(z)$ for $a=-1$. According to calculations, the zeroes tend to a certain curve in the left half-plane that depends on the limiting ratio~$\dfrac{k}{n}$ as $k\to\infty$, $n>k$.
\begin{figure}[h]
\begin{multicols}{2}
\hfill
\includegraphics[scale=0.39]{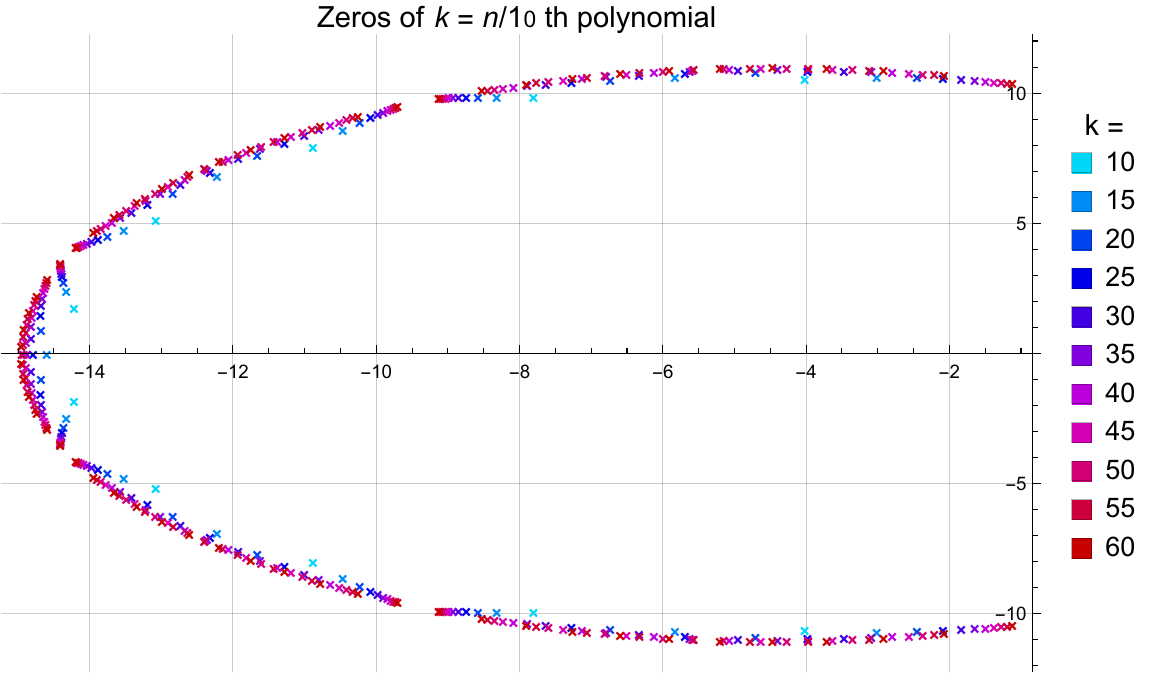}
\hfill\hfill\hfill\hfill
\caption{Zeroes of $P_{k}$ for $a=-1$, $n=10k$, and $k=10,\ldots,60$}\label{Pic.a}
\hfill
\includegraphics[scale=0.39]{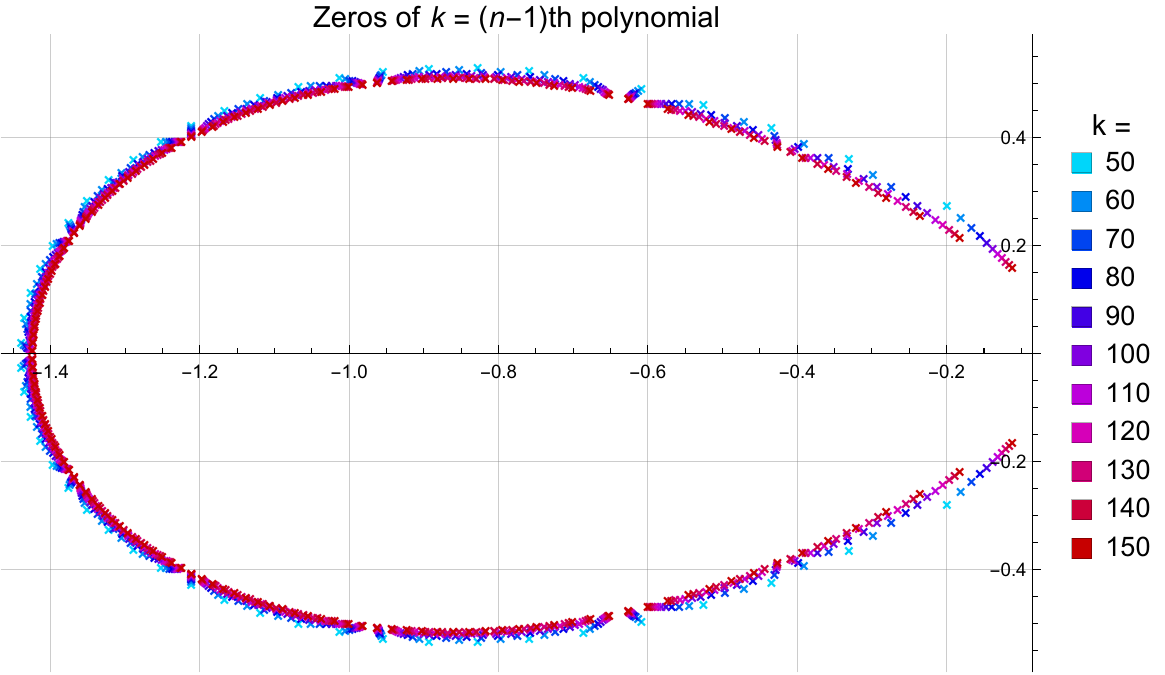}
\hfill\hfill
\caption{Zeroes of $P_{k}$ for $a=-1$, $n=k+1$, and $k=50,\ldots,150$}\label{Pic.ax}
\end{multicols}
\end{figure}

If one changes linear dependence between~$n$ and~$k$ to higher powers, the zeroes of~$P_k$ move to infinity as~$k\to+\infty$, see Figures~\ref{Pic.a.12} and~\ref{Pic.ax.45}.
\begin{figure}[h]
\begin{multicols}{2}
\hfill
\includegraphics[scale=0.39]{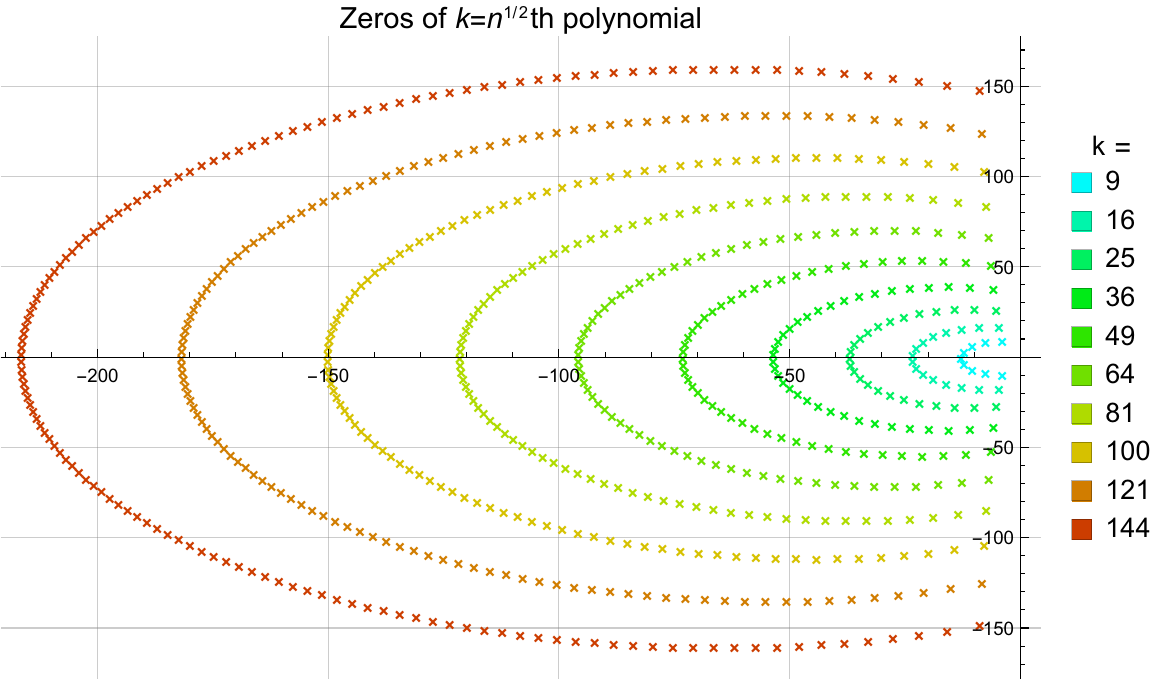}
\hfill\hfill\hfill\hfill
\caption{Zeroes of $P_{k}$ for $a=-1$, $n=k^2$, and $k=9,\ldots,144$}\label{Pic.a.12}
\hfill
\includegraphics[scale=0.39]{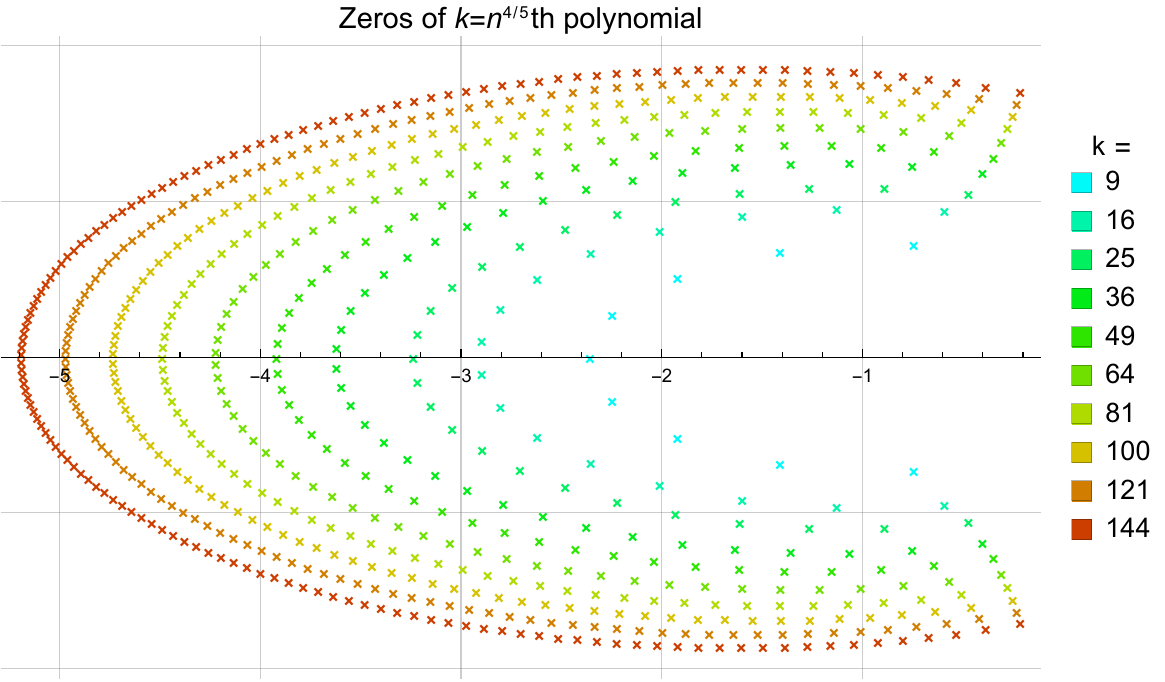}
\hfill\hfill
\caption{Zeroes of $P_{k}$ for $a=-1$, $n= \lfloor k^{5/4}\rfloor$, and $k=9,\ldots,144$}\label{Pic.ax.45}
\end{multicols}
\end{figure}
This phenomenon may be explained as follows. As $\dfrac kn\to0$ we can rescale the polynomials $P_k(z)$ to have the zeroes bounded. The proper rescaling may be achieved by letting $a=-\dfrac1n$ and study the asymptotic behaviour of the zeroes of
$P_k\left(\dfrac zk\right)$, which is equivalent to letting~$a=-\dfrac kn$ and studying the zeroes of~$P_k(z)$, see Figures~\ref{Pic.a.12s} and~\ref{Pic.ax.45s}. From the discussion in Section~\ref{section:Bessel} it follows
that the zeroes of $P_k\left(\dfrac zk\right)$ for $a=-\dfrac1n$ should approximate those
of $B_k\left(\dfrac zk\right)$ as $k\to\infty$, $n\to\infty$, $\dfrac kn\to0$. At the same time, the latter zeroes are known
to remain bounded~\cite{Grosswald.1978,Olver.1954}, see also~\cite{deBruin.Saff.Varga.1,deBruin.Saff.Varga.2}.

\begin{figure}[h]
\begin{multicols}{2}
\hfill
\includegraphics[scale=0.39]{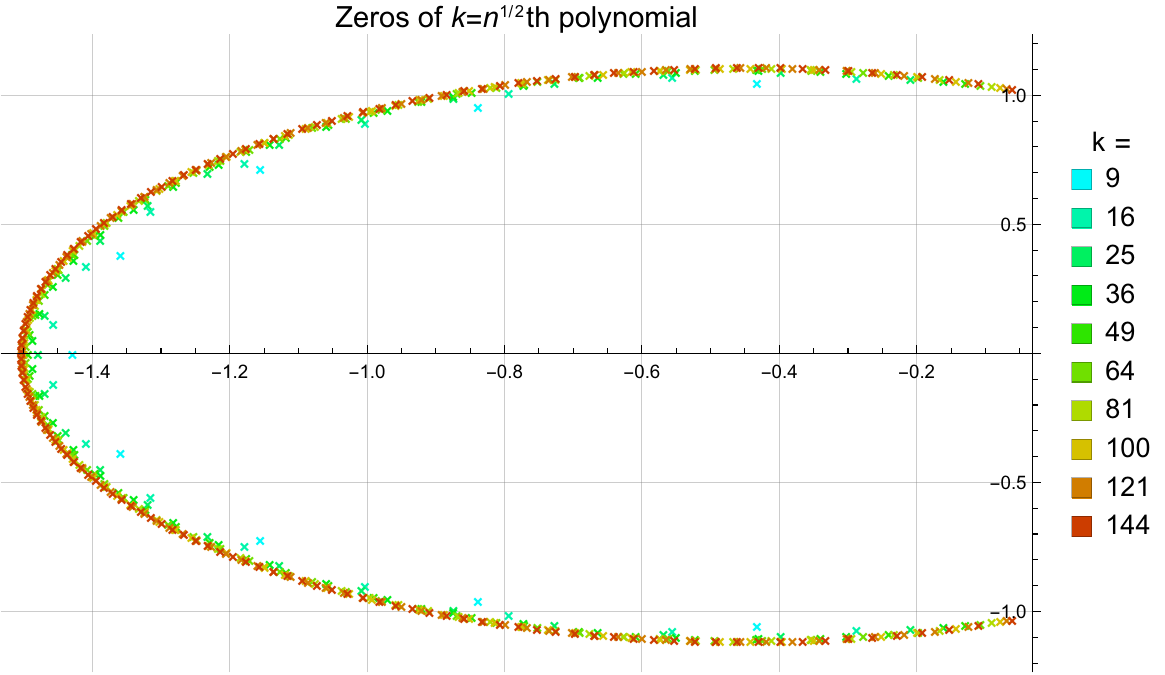}
\hfill\hfill\hfill\hfill
\caption{Zeroes of $P_{k}$ for $a=-k/n$, $n=k^2$, and $k=9,\ldots,144$}\label{Pic.a.12s}
\hfill
\includegraphics[scale=0.39]{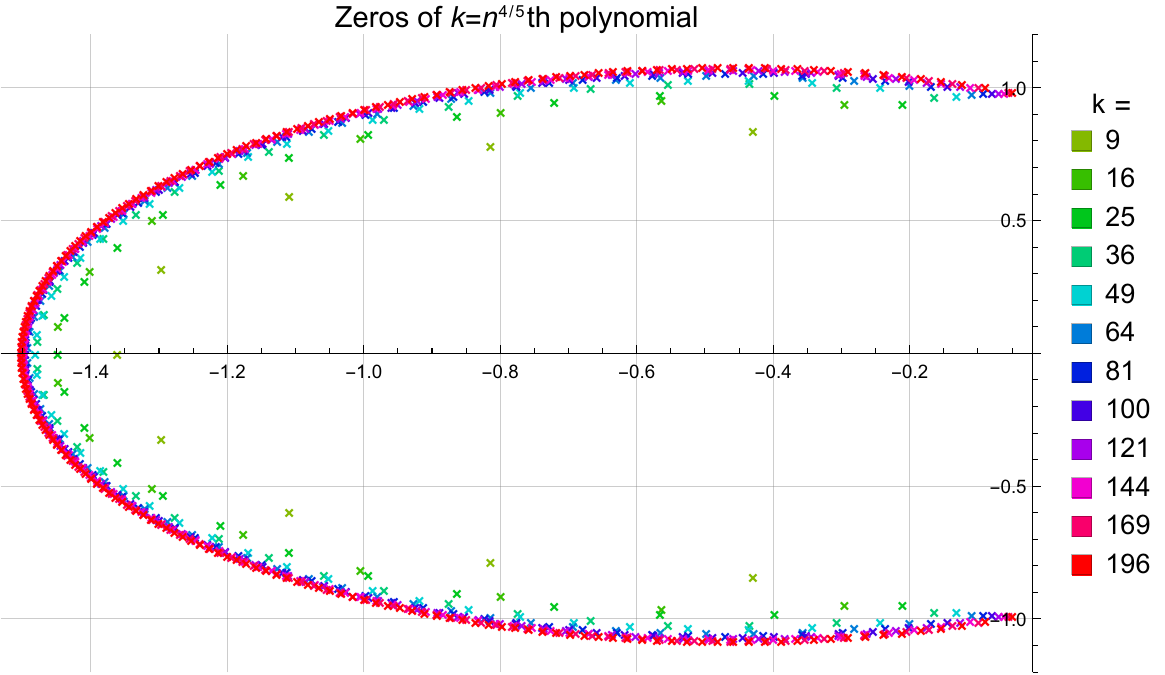}
\hfill\hfill
\caption{Zeroes of $P_{k}$ for $a=-k/n$, $n= \lfloor k^{5/4}\rfloor$, and $k=9,\ldots,144$}\label{Pic.ax.45s}
\end{multicols}
\end{figure}

\setcounter{equation}{0}
\section*{Acknowledgement}


The results of Theorems~\ref{Theorem.rat.func} and~\ref{Theorem.Hankel.minors.exact}, Corollary~\ref{Corol.b_k} and Sections~\ref{subsection:orthogonality} and~\ref{section:Bessel} (except Theorem~\ref{Thm.sec.3}) were obtained with the support of the Russian Science Foundation grant~19-71-30002. The work on Remarks~\ref{Remark:moments} and~\ref{Remark:Pascal}, Section~\ref{sect:P_k.basic.properties} and all numerical calculations were supported by the state assignment, registration number 122041100132-9.

\end{document}